\documentclass[10pt]{article}

\usepackage{tikz}
\usepackage{tikz-cd}
\usepackage{witharrows}
\usepackage{graphicx,wrapfig,lipsum}   

\usepackage{xspace}

\newcommand{\term}[1]{\text{\tt{#1}}\xspace}

\newcommand{\proofstep}[1]{%
  \par
  \addvspace{\medskipamount}
  \textit{#1
  }\enspace\ignorespaces
}

\usepackage{amsmath,amssymb}
\usepackage{amsthm}
\usepackage{mathtools}
\usepackage[noend]{algorithmic}
\usepackage[ruled,vlined]{algorithm2e}
\usepackage{url}
\usepackage{fullpage}
\usepackage{makeidx}
\usepackage{enumerate}
\usepackage[top=1in, bottom=1in, left=1in, right=1in]{geometry}
\usepackage{graphicx,float,psfrag,epsfig,caption,subcaption}

\usepackage{hyperref}
\usepackage{epstopdf}
\usepackage{color}
\usepackage{xr}

\usepackage{caption}
\usepackage[utf8]{inputenc}
\usepackage{thm-restate}

\usepackage{scalerel,stackengine}

\usepackage{enumitem}

\stackMath
\newcommand\reallywidehat[1]{%
\savestack{\tmpbox}{\stretchto{%
  \scaleto{%
    \scalerel*[\widthof{\ensuremath{#1}}]{\kern.1pt\mathchar"0362\kern.1pt}%
    {\rule{0ex}{\textheight}}
  }{\textheight}%
}{2.4ex}}%
\stackon[-6.9pt]{#1}{\tmpbox}%
}
\usepackage[mathscr]{euscript}

\DeclareSymbolFont{rsfs}{U}{rsfs}{m}{n}
\DeclareSymbolFontAlphabet{\mathscrsfs}{rsfs}

\numberwithin{equation}{section}

\newtheoremstyle{myexample} 
    {\topsep}                    
    {\topsep}                    
    {\rm }                   
    {}                           
    {\bf }                   
    {.}                          
    {.5em}                       
    {}  

\newtheoremstyle{myremark} 
    {\topsep}                    
    {\topsep}                    
    {\rm}                        
    {}                           
    {\bf}                        
    {.}                          
    {.5em}                       
    {}  

\newtheorem{claim}{Claim}[section]
\newtheorem{lemma}[claim]{Lemma}

\newtheorem{theorem}{Theorem}
\newtheorem{proposition}[claim]{Proposition}

\newtheorem{definition}[claim]{Definition}

\theoremstyle{myremark}
\newtheorem{remark}{Remark}[section]

\theoremstyle{myremark}

\theoremstyle{myexample}

\newcommand{\stable}{\term{stable}}

\DeclareMathOperator{\bp}{\mathbf{p}}
\DeclareMathOperator{\bq}{\mathbf{q}}

\newcommand{\M}{\term{max}}
\newcommand{\Min}{\term{min}}

\newcommand{\E}{\mathbb E}
\newcommand{\lex}{\term{lex}}
\newcommand{\tot}{\term{tot}}
\newcommand{\for}{\term{forward}}
\newcommand{\back}{\term{backward}}

\newcommand{\Blocks}{\term{Blocks}}

\newcommand{\Pre}{\term{Pre}_{\term{CL}}}
\newcommand{\CL}{\term{CL}}
\DeclareMathOperator{\Sym}{\mathfrak{S}}
\newcommand{\Bin}{\term{Bin}}
\newcommand{\cL}{c_{L}}
\newcommand{\cP}{c_{F}}
\newcommand{\cW}{c_{E}}
\newcommand{\cE}{c_{X}}
\newcommand{\cD}{c_{D}}

\usepackage{framed}
\title{Cutoff for the Asymmetric Riffle Shuffle}
\author{
Mark Sellke\thanks{
        Department of Mathematics,
        Stanford University.
    }
}
\date{}

\begin{document}

\maketitle

\begin{abstract}
\noindent
In the Gilbert-Shannon-Reeds shuffle, a deck of $N$ cards is cut into two approximately equal parts which are riffled together uniformly at random. Bayer and Diaconis \cite{bayer1992trailing} famously showed that this Markov chain undergoes cutoff in total variation after $\frac{3\log(N)}{2\log(2)}$ shuffles. We establish cutoff for the more general \emph{asymmetric} riffle shuffles in which one cuts the deck into differently sized parts. The value of the cutoff point confirms a conjecture of \cite{lalley2000rate}. Some appealing consequences are that asymmetry always slows mixing and that total variation mixing is strictly faster than separation and $L^{\infty}$ mixing. 
\end{abstract}

{\small \tableofcontents}

\section{Introduction}

The riffle shuffle is among the most common methods to randomize a deck of cards. We study a parameterized model for riffle shuffles called $p$-shuffles, defined as follows for any $p\in (0,1)$. From a sorted deck of $N$ cards, first remove the top $\Bin(N,p)$ cards to create a top and a bottom pile. Next, interleave the two piles according to the following rule. If the piles currently have sizes $A$ and $B$, the next card is dropped from the first pile with probability $\frac{A}{A+B}$. Conditioned on the pile sizes, this rule gives a uniformly random interleaving.

The case $p=\frac{1}{2}$, known as the Gilbert-Shannon-Reeds (GSR) shuffle, is perhaps the most natural model for riffle shuffling. It was analyzed by Bayer and Diaconis in \cite{bayer1992trailing} following work of Aldous (\cite[Example 4.17]{aldous1983random}); they proved that $\left(\frac{3}{2\log(2)}\pm o(1)\right)\log(N)$ shuffles are necessary and sufficient to randomize a deck. More precisely for any $\varepsilon>0$, as $N\to\infty$ the total variation distance of the deck from a uniform permutation tends to $1$ after $\left\lfloor\left(\frac{3}{2\log(2)}-\varepsilon\right)\log(N)\right\rfloor$ shuffles, and tends to $0$ after $\left\lfloor\left(\frac{3}{2\log(2)}+\varepsilon\right)\log(N)\right\rfloor$ shuffles. In fact they showed that the total variation distance decays exponentially in $C$ after $\frac{3\log(N)}{2\log(2)}+C$ shuffles.

By contrast, determining the mixing time for general $p$-shuffles has remained open. This discrepancy is because of a special property underpining the analysis in \cite{bayer1992trailing}: the deck order after a fixed number of GSR shuffles is uniformly random conditioned on its number of \emph{rising sequences}. Therefore to understand the mixing time it suffices to understand how the number of rising sequences is distributed. This distribution turns out to admit a simple closed form, which enables explicit analysis and a sharp understanding of the rate of convergence. When $p\neq \frac{1}{2}$ this conditional uniformity no longer holds and the problem becomes more complicated.

$p$-shuffles were introduced in \cite[Example 7]{diaconis1992analysis} and further studied in \cite{lalley1996cycle,fulman1998combinatorics,lalley2000rate}. These works established upper and lower bounds of order $\log(N)$ on the mixing time, but with differing constant factors. Interestingly the eigenvalues of the $p$-shuffle chain are given explicitly by certain power sum symmetric functions. This follows from general results regarding random walks on hyperplane arrangements --- see \cite{bidigare1999combinatorial,brown1998random,stanley2001generalized} or the survey \cite{zhao2009biased}.

Several aspects of riffle shuffles are surveyed in \cite{diaconis2003mathematical}. Other interesting models arise from modifying the interleaving probabilities, such as the Thorpe shuffle \cite{thorp1973nonrandom,morris2009improved,morris2013improved} and clumpy shuffle \cite{jonasson2015rapid}.

\subsection{Main Result}

In this paper we prove that all $p$-shuffles exhibit cutoff. More generally, let $\bp=(p_0,\dots,p_{k-1})$ be a discrete probability distribution with $p_i>0$ for each $i$. We show cutoff for the more general $\bp$-shuffles, which were also introduced in \cite{diaconis1992analysis}. To define such a shuffle, one first generates a multinomial $(N,\bp)$ vector $(n_0,\dots, n_{k-1})$ so that each $n_i$ has marginal distribution $n_i\sim \Bin(N,p_i)$ and $\sum_{i=0}^{k-1} n_i=N$ holds almost surely. One then splits the $N$ cards into $k$ piles by taking the top $n_0$ cards off the top to form the first pile, the next $n_1$ cards to form the second pile, and so on.

Interleaving the $k$ piles into a single pile is done similarly to the $k=2$ case. 
Namely, if the current remaining pile sizes are $A_0,\dots,A_{k-1}$, then the next card is dropped from pile $i$ with probability
\begin{equation}
\label{eq:interleave}
    \frac{A_i}{A_0+A_1+\dots+A_{k-1}}.
\end{equation}
This latter phase is again equivalent to interleaving the $k$ piles uniformly at random conditioned on their sizes. 
Note that the asymmetry of $\bp$ appears only in the first phase to determine the pile sizes and does not directly enter the second phase.
When $\bp=\left(\frac{1}{k},\frac{1}{k},\dots,\frac{1}{k}\right)$, we recover the $k$-shuffle which is the $k$-partite analog of the GSR shuffle. $k$-shuffles exhibit cutoff after $\frac{3\log(N)}{2\log k}\pm O(1)$ steps by the same rising sequence analysis as in the $k=2$ case (\cite{bayer1992trailing}).

To state our main result for general $\bp$-shuffles, we must define several constants. With arbitrary tie-breaking, set $i_{\M}=\arg\max_{i\in \{0,1,\dots,k-1\}}(p_i)$ and $p_{\M}=p_{i_{\M}}$. Similarly define $i_{\Min}$ and $p_{\Min}$. Define the functions 
\[
    \phi_{\bp}(t)=\sum_{i=0}^{k-1} p_i^t,\quad\quad \psi_{\bp}(t)=-\log\phi_{\bp}(t).
\]
Define the positive constant $\theta_{\bp}$ by the identity $\psi_{\bp}(\theta_{\bp})=2\psi_{\bp}(2)$, i.e. 
\[
\phi_{\bp}(\theta_{\bp})=\sum_{i=0}^{k-1} p_i^{\theta_{\bp}}=\left(\sum_{i=0}^{k-1} p_i^2\right)^2=\phi_{\bp}(2)^2.\] This uniquely determines $\theta_{\bp}$ because $\phi_{\bp}$ and $\psi_{\bp}$ are strictly monotone. Finally define the constants  $C_{\bp},\widetilde{C}_{\bp},$ and $\overline{C}_p$ as follows.
\begin{align*}
    C_{\bp}&=\frac{3+\theta_{\bp}}{4\psi_{\bp}(2)}=\frac{3+\theta_{\bp}}{2\psi_{\bp}(\theta_{\bp})},\\
    \widetilde C_{\bp}&=\frac{1}{\log(1/p_{\M})},\\
    \overline{C}_{\bp}&=\max(\widetilde C_{\bp},C_{\bp}).
\end{align*}
We can now state our main result.

\begin{theorem}\label{thm:main}

The $\bp$-shuffles undergo total variation cutoff after $\overline{C}_{\bp}\log(N)$ steps. That is, for any $\varepsilon>0$, 
\begin{align}
\label{eq:LB}
\lim_{N\to\infty} d_N(\lfloor(1-\varepsilon)\overline{C}_{\bp}\log(N)\rfloor)&=1,\\
\label{eq:UB}
\lim_{N\to\infty} d_N(\lfloor(1+\varepsilon)\overline{C}_{\bp}\log(N)\rfloor)&=0.
\end{align}
Here $d_N(K)$ denotes the total variation distance from uniform after $\bp$-shuffling $K$ times.

\end{theorem}

It is easy to see that $\overline{C}_{\bp}$ is symmetric and continuous in the entries of $\bp$. In the next proposition we show that for any $k$, the fastest possible mixing for any $\bp=(p_0,\dots,p_{k-1})$ occurs in the symmetric case $\bp=\left(\frac{1}{k},\frac{1}{k},\dots,\frac{1}{k}\right)$.

\begin{proposition}\label{prop:imp}

For any $k$, $\overline{C}_{\bp}$ has minimum value $\frac{3}{2\log k}$ achieved uniquely at $\bp=\left(\frac{1}{k},\frac{1}{k},\dots,\frac{1}{k}\right)$. Moreover for any $\bp$, 
\[
    C_{\bp}\in \bigg[\frac{3}{2\psi_{\bp}(2)},\frac{7}{4\psi_{\bp}(2)}\bigg)\quad\text{ and }\quad \widetilde{C}_{\bp}\in \bigg[\frac{1}{\psi_{\bp}(2)},\frac{2}{\psi_{\bp}(2)}\bigg).
\]
\end{proposition}

It also follows from Proposition~\ref{prop:imp} that for any $\bp$, cutoff occurs in total variation occurs strictly sooner than in the $L^{\infty}$ and separation distances. Quite precise results for these alternative notions of mixing are shown in \cite{assaf2011riffle} by different methods, for the same asymmetric riffle shuffles that we study. In particular, cutoff occurs in both of these distances after $\frac{2\log(N)}{\psi_{\bp}(2)}\pm O(1)$ shuffles. Recall that separation and $L^{\infty}$ distance always upper-bound total variation distance, so only the strictness of the resulting inequality
\[
    \frac{2}{\psi_{\bp}(2)}>\overline{C}_{\bp}
\] 
between mixing time growth rates is non-trivial.

\begin{proof}[Proof of Proposition~\ref{prop:imp}]
When $\bp=\left(\frac{1}{k},\frac{1}{k},\dots,\frac{1}{k}\right)$ it is easy to see that $\theta_{\bp}=3$ and $\phi_{\bp}(2)=\frac{1}{k}$. Therefore 
\[
C_{\bp}=\frac{3}{2\log k}>\frac{1}{\log k}=\widetilde{C}_{\bp}.
\]
The value $\phi_{\bp}(2)$ is symmetric and strictly convex in $\bp$, hence achieves unique minimum at $\bp=\left(\frac{1}{k},\frac{1}{k},\dots,\frac{1}{k}\right)$. Moreover $\theta_{\bp}\geq 3$ always holds as Cauchy--Schwarz implies
\[
    \phi_{\bp}(2)^2
    =\left(\sum_{i=0}^{k-1} p_i^2\right)^2
    \leq \left(\sum_{i=0}^{k-1} p_i^3\right)\cdot \left(\sum_{i=0}^{k-1} p_i\right)
    =\sum_{i=0}^{k-1} p_i^3=\phi_{\bp}(3).
\]
Therefore $C_{\bp}$ achieves unique minimum at $\bp=\left(\frac{1}{k},\frac{1}{k},\dots,\frac{1}{k}\right)$, hence the first result. Moreover $\theta_{\bp}< 4$ also holds because
\[\phi_{\bp}(2)^2=\left(\sum_{i=0}^{k-1} p_i^2\right)^2 > \sum_{i=0}^{k-1} p_i^4=\phi_{\bp}(4).
\]
This shows that $C_{\bp}\in \bigg[\frac{3}{2\psi_{\bp}(2)},\frac{7}{4\psi_{\bp}(2)}\bigg)$. It remains to estimate $\widetilde{C}_{\bp}$, and the claimed bounds amount to showing  \[\sum_{i=0}^{k-1} p_i^2\leq p_{\M}<\sqrt{\sum_{i=0}^{k-1} p_i^2}.\] The left inequality holds because  \[\sum_{i=0}^{k-1} p_i^2\leq \sum_{i=0}^{k-1} p_ip_{\M}= p_{\M}\] and the right inequality is clear.
\end{proof}

The primary focus of this paper is showing the upper bound \eqref{eq:UB}, i.e. that the mixing time is at most $\overline{C}_{\bp}\log(N)$. In Section~\ref{sec:UB} we reduce \eqref{eq:UB} to the estimation of a certain exponential moment, which occupies Sections~\ref{sec:proof} and \ref{sec:UBfinish}. In the other direction, Lalley showed mixing time lower bounds of both $\widetilde C_{\bp}\log(N)$ and $C_{\bp}\log(N)$ in \cite{lalley2000rate}. However the latter result required $\bp\approx\left(\frac{1}{k},\frac{1}{k},\dots,\frac{1}{k}\right)$ to be close to uniform. (\cite{lalley2000rate} only considered the case $k=2$, but the arguments work identically for larger $k$.) In Section~\ref{sec:lb} we generalize the $C_{\bp}\log(N)$ lower bound to all $\bp=(p_0,\dots,p_{k-1})$ by refining Lalley's approach. For the sake of continuity, several of our notational choices, such as the constants $C_{\bp}$ and $\widetilde{C}_{\bp}$, are adopted from \cite{lalley2000rate}. However we reversed the sign of $\psi_{\bp}$ from \cite{lalley2000rate} so that $\psi_{\bp}(t)>0$ for all $t>1$.

\begin{table}[H]
\large
\centering
\begin{tabular}{ |p{1.8cm}||p{1.7cm}|p{1.7cm}|p{1.7cm}| p{1.6cm}|p{1.6cm}|p{1.8cm}|  }
 \hline
 \multicolumn{7}{|c|}{Approximate Mixing Times $\overline{C}_{\bp} \log N$ for $p$-Shuffles } \\
 \hline 
  Deck Size & $p=0.5$ & $p=0.6$ & $p=0.7$& $p=0.8$ & $p=0.9$ & $p=0.95$ \\
 \hline
 52& 8.6 &9.2 &11.3 &18 &37 &77\\
 104& 10.1 &10.8 &13.3 &21 &44 &90 \\
 208 & 11.6 &12.4 &15.3 &24 &51 &104 \\
 520    & 13.5 &14.5 &17.9 &28 &59 &122 \\
 $N$   &  $2.16\log N$ &  $2.32\log N$   & $2.86\log N$  &  $4.5\log N$  &  $9.5\log N$ &  $19.5\log N$\\
 \hline
\end{tabular}
\caption{\large The values $\overline{C}_{\bp}\log N$ are shown for varying deck sizes $N$ and $\bp=(p,1-p)$. These values should be taken as a rough guide because our results are asymptotic in $N$.}
\end{table}


\vspace{-1cm}

\begin{figure}[H]
\centering
\includegraphics[width=\linewidth]{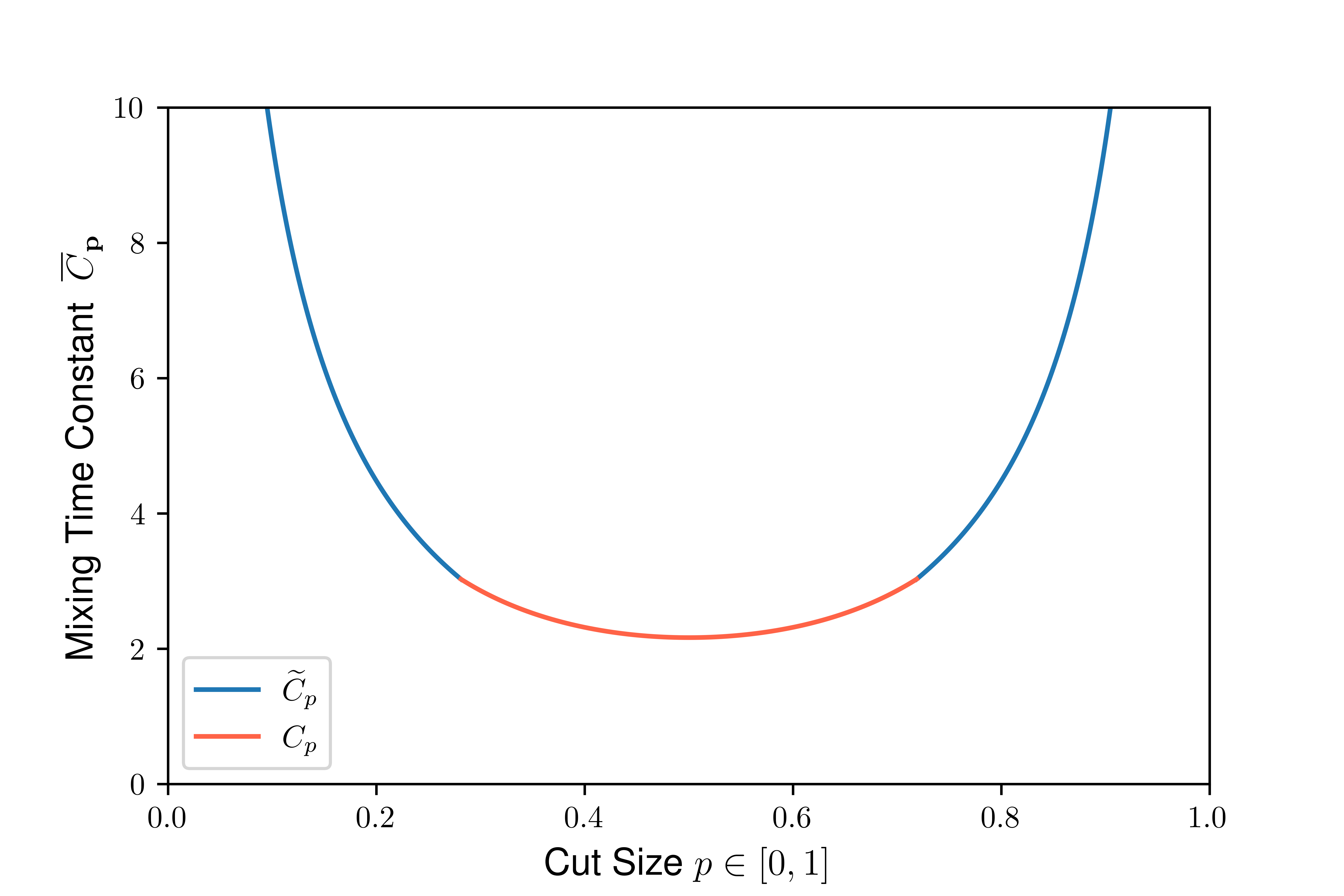}
\caption{The values $\overline{C}_{\bp}$ for $\bp=(p,1-p)$ are shown. The blue and red depict the transitions between $C_{\bp}$ and $\widetilde C_{\bp}$, which occur at $p\approx 0.28$ and $p\approx 0.72$. As $p\to 0$, the divergence is $\overline{C}_{\bp}=\frac{1}{\log(1/(1-p))}=\frac{1}{p}+O(1)$.}
\label{fig:test}
\end{figure}

\section{Preliminaries}

Let $P_{\bp}$ denote the probability measure on the symmetric group $\Sym_N$ given by applying a $\bp$-shuffle to the identity. Given two discrete probability vectors $\bp=(p_0,\dots,p_{k-1})$ and $\mathbf{q}=(q_0,\dots,q_{\ell-1})$ define their convolution 
\[
    \bp*\bq
    \equiv (p_0q_0,p_0q_1,\dots,p_0q_{\ell-1},p_1q_0,\dots,p_{k-1}q_{\ell-1}).
\] 
This convolution turns out to correspond to shuffle composition.

\begin{proposition}[{\cite[Example 7]{diaconis1992analysis}}]
\label{prop:basic}
Performing a $\mathbf{q}$-shuffle followed by a $\bp$-shuffle is equivalent to performing a $(\bp*\bq)$-shuffle. That is, 
\[
    P_{\bp}*P_{\bq}=P_{\bp*\bq}.
\]
\end{proposition}

Proposition~\ref{prop:basic} yields an explicit description for the distribution $P_{\bp^{*K}}$ of a deck after $K$ shuffles. For instance in the ``symmetric'' setting of \cite{bayer1992trailing}, it implies that composing a $k_1$-shuffle and a $k_2$-shuffle results in a $k_1k_2$-shuffle. It will actually be more convenient for us to work with the inverse permutations. We now explain how to do this, following \cite{lalley2000rate}. First define a distribution on sequences 
\[
    S=(s_1,\dots, s_N)
\]
of length $K$ strings as follows. Generate $N$ strings of length $K$, all with i.i.d. $\bp$-random digits in 
\[
    [k]_0=\{0,\dots,k-1\}.
\] 
$S$ is obtained by sorting these strings into increasing lexicographic order 
\[
    s_1\leq_{\lex}s_2\leq_{\lex}\dots\leq_{\lex}s_N.
\]
Recall that the lexicographic order on strings of the same length is just given by comparing their base $k$ values. In general, the lexicographically smaller of two different $[k]_0$-strings is the one with the smaller digit at the first place where their digits differ, or is the shorter string if one string is a prefix of the other.

Next define the associated \emph{shuffle graph} $G=G(S)$ on vertex set 
\[
    [N]=\{1,2,\dots,N\}
\]
in which $i,i+1\in V(G)$ are neighbors if and only if $s_i=s_{i+1}$, and no other edges are in $G$. Hence $G$ is a union of disjoint paths, which we call \emph{$G$-components}. 
(We say $S$ and $G=G(S)$ are $\bp$-random when they are constructed in this way.)
Finally choose a uniformly random permutation $\pi\in\Sym_N$ and define its $G$-modification $\pi^G$ by, within each $G$-component, sorting the values $\pi(i)$ into increasing order. The next proposition states that $\pi^G$ is exactly the inverse permutation of a $\bp^{*K}$-shuffled deck.

\begin{figure}
\centering
\frame{
\includegraphics[width=\linewidth]{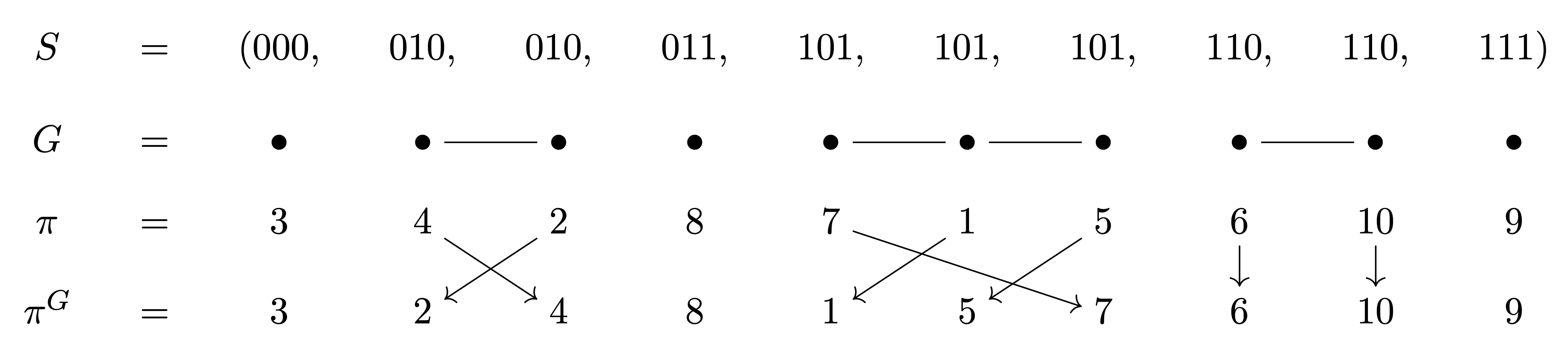}
}
\caption{In this example with $N=10$ strings in $[k]_0^K=[2]_0^3$, the lexiographically sorted sequence of strings $S$ leads to the shuffle graph $G=G(S)$. The permutation $\pi\in \Sym_N$ is then transformed into $\pi^G$ by sorting within each $G$-component. By Proposition~\ref{prop:corr}, the inverse $(\pi^G)^{-1}$ of the resulting permutation has distribution $P_{\bp^{*K}}$.}

\end{figure}

\begin{proposition}[{\cite[Lemma 3]{lalley2000rate}}]\label{prop:corr} 
Let $\pi\in\Sym_N$ be uniformly random and $G=G(S)$ be $\bp$-random as defined above, for some fixed positive integers $N$ and $K$. Then the distribution of $(\pi^G)^{-1}$ is exactly $P_{\bp^{*K}}$. In particular, the total variation distance of $\pi^G$ from uniform equals $d_N(K)$.
\end{proposition}

In other words, the inverse permutation of a shuffled deck can be generated by starting with a uniformly random permutation $\pi$, and then modifying $\pi$ to create $\pi^G$ which is increasing on an independently random set of subintervals in $[N]$. After more and more shuffles, these subintervals shrink in distribution, leading eventually to mixing. In fact, $L^{\infty}$ and separation mixing both correspond to $G$ having no edges with high probability, see \cite[Corollary 3]{lalley2000rate} and \cite{assaf2011riffle}. However because $G$ is random, total variation mixing can and does occur sooner. We refer the reader to \cite[Section 2]{lalley2000rate} for more explanation and examples regarding Proposition~\ref{prop:corr}. In brief, the $N$ sequences $s_i\in [k]_0^K$ correspond to the sequences of pile-types that each of the $N$ cards in the deck appears in during the shuffles. The sorting within $G$-components corresponds to the fact that if two cards are in the same pile during all $K$ of the riffle shuffles, then their relative order must be preserved.

Throughout the remainder of this paper, we work \textbf{entirely} with this transformed problem. Namely we will show that for $K\geq (1+\varepsilon)\overline{C}_p\log N$ the permutation $\pi^G$ has total variation distance $o(1)$ from uniform, while for $K\leq (1-\varepsilon)\overline{C}_p\log N$ this distance is $1-o(1)$.


\subsection{Intuition Based on an Independent Point Process}\label{subsec:intuition}

There are two main obstructions to mixing which lead to the separate lower bounds of $\widetilde C_{\bp}$ and $C_{\bp}$. The simpler obstruction is that if $K\leq (\widetilde C_{\bp}-\varepsilon)\log(N)$, then some strings will typically occur many times, so $\pi^G$ will contain an abnormally long increasing substring of length $N^{\Omega(1)}$. Indeed, from the definition $\widetilde C_{\bp}=\frac{1}{\log(1/p_{\M})}$ it follows that after $K\leq (\widetilde C_{\bp}-\varepsilon)\log(N)$ shuffles, the expected number of strings with $s_j=i_{\M}^K$ is 
\begin{align*}
    \E \left|\{j\in [N]:s_j=i_{\M}^K\}\right|&= p_{\M}^K N \\
    &\geq N^{-(\widetilde C_{\bp}-\varepsilon)\log(1/p_{\M})+1}\\
    &\geq N^{\Omega_{\varepsilon}(1)}.
\end{align*}
Since the number of such strings is binomially distributed, it is well-concentrated around its mean. Therefore with probability $1-o(1)$ the $\bp$-random shuffle graph $G$ contains a length $N^{\Omega_{\varepsilon}(1)}$ path, and so $\pi^G$ contains an increasing contiguous substring of the same length. However in a uniformly random permutation $\pi$, the probability to have an increasing substring of length $\ell\geq \log N$ is at most $N/(\ell!)=o(1)$. Therefore the total variation distance from uniform is $1-o(1)$ when $K\leq (\widetilde C_{\bp}-\varepsilon)\log(N)$.

The more complicated obstruction to mixing comes from a fractal set of predictable locations (referred to as ``cold spots'' in \cite{lalley2000rate}) which tend to contain many $G$-edges. This obstruction, as well as our approach to the upper bound, can be motivated by an independent point process heuristic. (See also the last section of \cite{lalley2000rate}.) Suppose we observe $\sigma\in\Sym_N$ which is generated by either $\sigma=\pi$ or $\sigma=\pi^G$ for uniformly random $\pi\in \Sym_N$ and $\bp$-random $G$. Since the transformation $\pi\to\pi^G$ simply arranges small subintervals into increasing order, let us suppose that we observe only the ascent set $A(\sigma)=\{i:\sigma(i)<\sigma(i+1)\}$.
As a heuristic, we may treat $A(\sigma)$ as an independent point process on edges in both the uniform $\sigma=\pi$ and shuffled $\sigma=\pi^G$ distributions.
Specifically, for each $i\in [N-1]$ let
\[
    \eta_i\equiv\mathbb P[(i,i+1)\in E(G)].
\]
be the probability for $(i,i+1)$ to be an edge in $G$. Then 
\[
    \mathbb P[(i,i+1)\in A(\pi)]=\frac{1}{2}
\] 
while, roughly speaking,
\[
    \mathbb P[(i,i+1)\in A(\pi^G)]\approx \frac{1+\eta_i}{2}.
\]
(Technically $\mathbb P[(i,i+1)\in A(\pi^G)]$ should also depend on $\eta_{i-1}$ and $\eta_{i+1}$ but we will ignore this point.) This heuristic suggests that the likelihood ratio 
\[
    \frac{\mathbb P^{\pi\in\Sym_N}[\pi^G=\sigma]}{\mathbb P^{\pi\in\Sym_N}[\pi=\sigma]}
\]
evaluated at a uniformly random $\sigma\in\Sym_N$ behaves like the random product
\[
\prod_{i\in [N-1]} (1\pm \eta_i)
\]
where the $\pm$ signs are i.i.d. uniform. This product is close to $0$ in probability (so mixing has not occured) if $\sum_i \eta_i^2\gg 1$, and is close to $1$ in probability (so mixing has occured) if $\sum_i \eta_i^2\ll 1$. 

Next observe that even without heuristic assumptions, 
\[
    \sum_i \eta_i^2=\mathbb E[E(G,G')]
\] 
is the expected size of the edge-intersection 
\[
    E(G,G')\equiv E(G)\cap E(G')
\] 
of two independent $\bp$-random shuffle graphs $G$ and $G'$. Therefore it is natural to guess that mixing occurs once $|E(G,G')|$ is typically small. Indeed, the quantity $|E(G,G')|$ will be crucial throughout. Let us finally summarize how it and related quantities appear in the proofs.

To lower bound the mixing time, one identifies deterministic ``cold spot'' sets $H\subseteq [N]$ which typically contain at least $|H|^{\frac{1}{2}+\delta}$ $G$-edges and shows that this implies non-mixing (see Proposition~\ref{prop:lowerboundcond}). The existence of such sets $H$ implies in general that $\E[|E(G,G')|]\gg 1$ (Remark~\ref{rem:LB1}). Moreover in the independent point process model, the existence of such sets $H$ is essentially equivalent to $\sum_i \eta_i^2\gg 1$. Indeed, if $\sum_i \eta_i^2\gg N^{\delta}$ then by the dyadic pigeonhole principle it follows that for some positive integer $n$ there are at least $\Omega(2^{2n}N^{\delta/3})$ values $i\in [N-1]$ with $\eta_i\in [2^{-n},2^{-n+1}]$. These values of $i$ can be taken for the set $H$.

On the other hand, it can happen that $\E[|E(G,G')|]\ll 1$ holds strictly before the onset of total variation mixing. This requires that $p_{\M}>\max(p_0,p_{k-1})$ and in particular $k\geq 3$ --- see Remark~\ref{rem:gap}. Instead as explained in Section~\ref{sec:UB}, we reduce the mixing time upper bound \eqref{eq:UB} to showing that suitably truncated \textbf{exponential} moments of $|E(G,G')|$ are small.
Estimating these exponential moments is rather involved. Our strategy is outlined just before the start of Subsection~\ref{subsec:lemmas}, and the proof occupies Sections~\ref{sec:proof} and \ref{sec:UBfinish}.

These exponential moments arise naturally from considering (after some truncation) a chi-squared upper bound for total variation distance (see Lemma~\ref{lem:expbound}). In fact this seems to be a generally applicable method to upper-bound the total variation distance from a mixture of distributions with ``random hidden structure'' to a ``null distribution'' by controlling the interaction between two independent copies of the ``structure'' (in our setting, the graph $G$). For instance, a related observation was made in \cite[Proposition 3.2]{miller2012uniformity} and later exploited in \cite{lubetzky2016information,lubetzky2017universality} to analyze information percolation for the Ising model (see also \cite{lubetzky2015exposition}).

\subsection{Notation}

For any $M\geq 1$ the set $[k]_0^M$ consists of all length $M$ strings with digits in $[k]_0$. (All strings throughout the paper will have digits in $[k]_0$.) Let 
\[
    \mathcal S\subseteq ([k]_0^K)^N
\]
denote the set of all lexicographically non-decreasing sequences $S=(s_1,\dots,s_N)$ of $N$ strings with length $K$ each. Let $\mathcal G$ denote the set of all shuffle graphs, i.e. subgraphs of the path graph on $N$ vertices. For $G\in\mathcal G$, let $\mathcal C(G)$ be the set of $G$-components, i.e. connected components of $G$.

Define $\mu_{\bp,M}$, often abbreviated as just $\mu_{\bp}$, to be the probability measure on $[k]_0^M$ with each digit independently $\bp$-random. By abuse of notation, we also use $\mu_{\bp,M}$ or simply $\mu_{\bp}$ to denote the associated $\bp$-random distributions on $\mathcal S$ or $\mathcal G$. We sometimes use square brackets to denote strings written out by their digits. For instance $[(k-1)(k-1)]$ indicates the string with two digits of $(k-1)$ while $[(k-1)(k-1)0^{K-2}]$ denotes the string with two initial $(k-1)$-digits followed by $K-2$ final $0$-digits. We also occasionally use brackets to denote digits of a string, so for instance the digit expansion of a string $x$ may be written
\[
    x=x[1]x[2]\dots x[M]\in [k]_0^{M}.
\]

We write $\E^{\sigma},\E^{\pi},\mathbb P^{\sigma},$ and $\mathbb P^{\pi}$ to denote expectations or probabilities taken over uniformly random permutations $\sigma$ or $\pi$ in $\Sym_N$. We similarly write $\E^{S}$ to indicate expectation over $S\sim \mu_{\bp,K}$. We will continue to use $E(G,G')=E(G)\cap E(G')$ to denote the edge-intersection of $G,G'\in\mathcal G$. $S'$ and $G'=G(S')$ will always denote independent copies of $S$ and $G$.

The following definitions are used to prove Lemma~\ref{lem:connorepeat} at the end of Section~\ref{sec:UB}, and otherwise do not appear until Section~\ref{sec:proof}. For each string 
\[
    x=x[1]x[2]\dots x[M]\in [k]_0^{M}
\]
with $M\geq 1$ a positive integer, define
\begin{align}
\label{eq:tx} 
t_x&\equiv\mathbb P^{y\sim \mu_{\bp,M}}[y<_{\lex} x],\\ 
\label{eq:lambdax}
\lambda_x&\equiv\mathbb P^{y\sim \mu_{\bp,M}}[y=x]
=\prod_{i=1}^M p_{x[i]},\\
\label{eq:Jx}
J_x&\equiv[t_x,t_x+\lambda_x).
\end{align} 
Hence the intervals $(J_x)_{x\in [k]_0^M}$ partition $[0,1)$ for any fixed $M$. Note that 
\[
    t_x+\lambda_x=\mathbb P^{y\sim \mu_{\bp,M}}[y\leq_{\lex} x].
\]
It will often be useful to observe that to sample a $\bp$-random string $x\in [k]_0^M$, one may equivalently sample a uniform random variable $a\in [0,1]$ and take the unique $x$ with $a\in J_x$. Similarly to sample $(s_1,\dots,s_N)\in\mathcal S$, one may instead sample uniform i.i.d. 
\[
    a_1',\dots,a_N'\in [0,1],
\]
sort them into increasing order
\[
    0\leq a_1\leq\dots\leq a_N\leq 1,
\]
and finally choose $s_i\in [k]_0^K$ such that $a_i\in J_{s_i}$ for each $i\in [N]$.

\begin{figure}[H]
\centering
\begin{framed}
\includegraphics[width=\linewidth]{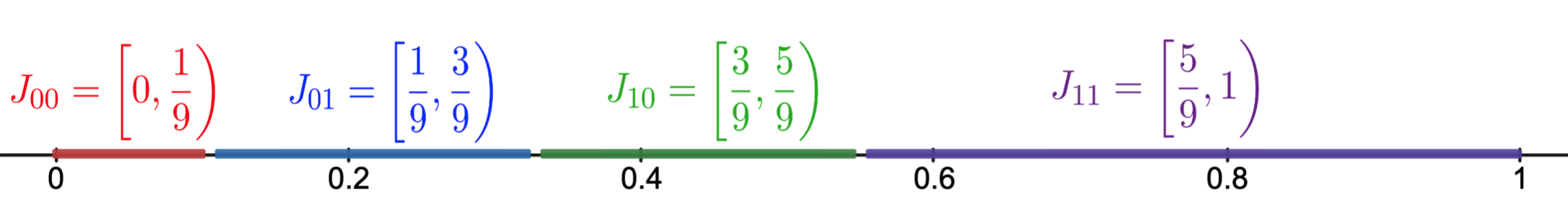}
\caption{The partition $[0,1)=\bigcup_{x\in[k]_0^{M}}J_x$ with $k=2,M=2$, and $(p_0,p_1)=\left(\frac{1}{3},\frac{2}{3}\right)$.}
\end{framed}
\label{fig:J}
\end{figure}

\section{Upper Bound Approach}\label{sec:UB}

Here we explain some of the ingredients used to prove the mixing time upper bound \eqref{eq:UB}. In Subsection~\ref{subsec:high-level} we present the more conceptual parts, ultimately reducing \eqref{eq:UB} to a certain exponential moment estimate. In Subsection~\ref{subsec:lemmas} we prove a few other lemmas used in Subsection~\ref{subsec:high-level}. This section might be viewed as an extended setup for the more difficult parts of the proof. For instance the constant $C_{\bp}$ does not explicitly enter until the next section. However we emphasize that the results of this section are both specific to the particular problem considered and essential to understand the remainder of the paper.

\subsection{High-Level Approach} \label{subsec:high-level}

We begin by carefully examining the Radon--Nikodym derivative between the distributions of $\pi^G$ and $\pi$ where $\pi$ is a uniformly random permutation. For each $G\in\mathcal G$, let $\mathcal C(G)=\{G_1,\dots,G_j\}$ be the $G$-components, and suppose that each $G_i$ contains $v_i$ vertices. Then it is easy to see that the map $\Sym_N\to\Sym_N$ given by $\pi\to \pi^G$ is $\left(\prod_{i=1}^j v_i!\right)$ to $1$. Moreover its image consists of those $\sigma$ with $\sigma^G=\sigma$. Therefore 
\[
    \mathbb P^{\pi}[\pi^G=\sigma]=1_{\sigma=\sigma^G}\cdot \frac{\prod_{i=1}^j v_i!}{N!},\quad \sigma\in \Sym_N.
\]
As a consequence, for fixed $G\in\mathcal G$ the Radon--Nikodym derivative $f_{G,\sigma}$ of $\pi^G$ with respect to $\pi$ is given by
\begin{align*}
    f_{G,\sigma}
    &\equiv 
    \frac{\mathbb P^{\pi}[\pi^G=\sigma]}{\mathbb P^{\pi}[\pi=\sigma]}
    \\
    &=
    N!\cdot \mathbb P^{\pi}[\pi^G=\sigma] 
    \\
    &=
    1_{\sigma^G=\sigma}\cdot \prod_{i=1}^j v_i!
    \\
    &= 
    \frac{1_{\sigma^G=\sigma}}{\mathbb P^{\pi}[\pi^G=\pi]}.
\end{align*}
Observe that for fixed $G\in\mathcal G$,
\begin{equation}\label{eq:EFG}
    \E^{\sigma}[f_{G,\sigma}]=1.
\end{equation}
On the other hand for fixed $\sigma$ and $\mu_{\bp,K}$-random $G=G(S)$, we may apply the law of total expectation to the second definition of $f_{G,\sigma}$ above. This implies that for fixed $\sigma$,
\[
    \mathbb P^{\pi,S}[\pi^{G(S)}=\sigma]=\frac{\E^{S}[f_{G(S),\sigma}]}{N!}.
\]
Therefore the total variation distance to uniform after $K$ shuffles is given by
\[
    d_N(K)=\frac{1}{2}\cdot \E^{\sigma}\left|\mathbb E^{S}[f_{G(S),\sigma}-1]\right|.
\]
Next, we use a chi-squared upper bound for total variation distance after removing exceptional sequences from $\mathcal S$. To carry this out, given a partition $\mathcal S=\mathcal S_1\cup \mathcal S_0$ (where $\mathcal S_1$ consists of ``typical'' sequences), write
\begin{align}
\nonumber 
    \mathbb E^{\sigma}\left|\mathbb E^{S}[f_{G(S),\sigma}-1]\right|
    &\leq 
    \mathbb E^{\sigma}\left|\mathbb E^{S}[(f_{G(S),\sigma}-1)1_{S\in \mathcal S_1}]\right|+\mathbb E^{\sigma}\left|\mathbb E^{S}[(f_{G(S),\sigma}-1)1_{S\in \mathcal S_0}]\right| 
    \\
\nonumber
    &\leq 
    \mathbb E^{\sigma}\left|\mathbb E^{S}[(f_{G(S),\sigma}-1)1_{S\in \mathcal S_1}]\right|
    +
    \mathbb E^{\sigma,S}[(f_{G(S),\sigma}+1)1_{S\in \mathcal S_0}]
    \\
\label{eq:2.1}
    &\stackrel{\eqref{eq:EFG}}{=}
    \mathbb E^{\sigma}\left|\mathbb E^{S}[(f_{G(S),\sigma}-1)1_{S\in \mathcal S_1}]\right|
    +
    2\mu_{\bp}(\mathcal S_0).
\end{align}
Take $S'$ to be an independent copy of $S$ and define for any shuffle graphs $G,G'\in\mathcal G$
\[
    f_{G,G'}\equiv \mathbb E^{\sigma}[f_{G,\sigma}f_{G',\sigma}].
\] 
We now use the Cauchy--Schwarz inequality to upper-bound the main term of \eqref{eq:2.1} via
\begin{align}
\left(\mathbb E^{\sigma}\left|\mathbb E^{S}[(f_{G(S),\sigma}-1)1_{S\in \mathcal S_1}]\right|\right)^2&\leq \mathbb E^{\sigma}\left[\left(\mathbb E^{S}[(f_{G(S),\sigma}-1)1_{S\in \mathcal S_1}]\right)^2\right]\nonumber\\
&=\mathbb E^{\sigma}\mathbb E^{S,S'}[(f_{G(S),\sigma}-1)(f_{G(S'),\sigma}-1)1_{S,S'\in \mathcal S_1}]\nonumber\\
&= \mathbb E^{\sigma}\mathbb E^{S,S'}[(f_{G(S),\sigma}f_{G(S'),\sigma}-1)1_{S,S'\in \mathcal S_1}]\nonumber\\
&=\mathbb E^{S,S'}\left[(f_{G,G'}-1)1_{S,S'\in \mathcal S_1}\right].\label{eq:2.2}
\end{align}

The second equality holds by switching the order of expectation and using \eqref{eq:EFG}. Starting from \eqref{eq:2.2} and throughout the remainder of the paper, we set $G=G(S),G'=G(S')$. Based on \eqref{eq:2.2}, to establish mixing it remains to show that $f_{G,G'}$ rarely exceeds $1$ in an $L^1$ sense (modulo choosing $\mathcal S_1$ and $\mathcal S_0$).

We will upper-bound $f_{G,G'}$ using the number $|E(G,G')|$ of edges shared by $G$ and $G'$. As motivation for why such a relationship should exist, observe that if no vertex $i\in [N]$ is incident to both a $G$-edge and a $G'$-edge, then $f_{G,\sigma}$ and $f_{G',\sigma}$ are \textbf{exactly} independent for $\sigma\in\Sym_N$ uniformly random. Hence in this case we have the exact equality
\[
    f_{G,G'}=\mathbb E^{\sigma}[f_{G,\sigma}f_{G',\sigma}]=\mathbb E^{\sigma}[f_{G,\sigma}]\mathbb E^{\sigma}[f_{G',\sigma}]
    \stackrel{\eqref{eq:EFG}}{=}1.
\]
In fact Lemma~\ref{lem:fbound} below implies that $f_{G,G'}\leq 1$ holds whenever $|E(G,G')|=0$. In essence, incident but non-overlapping edges only reduce $f_{G,G'}$. It is now unsurprising that $f_{G,G'}$ can be bounded above by some function of $|E(G,G')|$. We show in Lemma~\ref{lem:expbound} that this dependence is at most exponential when a condition called $L$-sparsity holds for both $G$ and $G'$. The requirement of $L$-sparsity will be part of the eventual definition of $\mathcal S_1$.

In general, for any shuffle graphs $G$ and $G'$, define the new shuffle graph $U$ to be their edge-union with $U$-components $\mathcal C(U)$. The next lemma shows how to estimate $f_{G,G'}$ based on the intersection structure of $G$ and $G'$. The proof is deferred to the next subsection.

\begin{lemma}\label{lem:fbound}

Suppose the $U$-components have vertex-sizes $(u_1,\dots,u_c)$. Then
\begin{equation}\label{eq:fggbound}
f_{G,G'}\leq \prod_{\substack{1\leq i\leq c,\\E(U_i)\cap E(G,G')\neq\emptyset}} (u_i!).
\end{equation}
\end{lemma}

We now define the first condition that ``typical'' sequences in $\mathcal S_1$ must satisfy. The objective is to ensure that the $u_i$ in Lemma~\ref{lem:fbound} are uniformly bounded by some constant $L=L(\bp,\varepsilon)$. Let us point out that it is not enough to argue that $\max_i(u_i)\leq L$ holds with high probability over random pairs $(S,S')$. Indeed, the truncation step \eqref{eq:2.1} was used to remove $\mathcal S_0$ before applying Cauchy--Schwarz to introduce $S'$. There is no analogous way to remove an arbitrary low-probability subset of \textbf{pairs} $(S,S')\in\mathcal S$. It is therefore important that the definition of $L$-sparsity below implies $\max_i(u_i)\leq L$ via separate restrictions on $G$ and $G'$.

\begin{definition}
For $L\geq 10$ a positive integer, a shuffle graph $G$ is $\boldsymbol{L}$\textbf{-sparse} if within any discrete interval $\{i,i+1,\dots,i+L-1\}\subseteq [N]$ of $L$ consecutive vertices, at most $L/3$ (of the possible $L-1$) edges are in $E(G)$.
\end{definition}

\begin{lemma}\label{lem:expbound}
Suppose $G$ and $G'$ are $L$-sparse shuffle graphs. Then $f_{G,G'}\leq (L!)^{|E(G,G')|}$.
\end{lemma}

\begin{proof}

We claim that $\max_i(u_i)\leq L$, i.e. each $U$-component contains at most $L$ vertices. Indeed by $L$-sparsity, $U$ contains at most $\frac{2L}{3}<L-1$ edges within each subinterval of $L$ vertices, hence no such interval can be a connected subgraph of $U$. Therefore Lemma~\ref{lem:fbound} implies that
\[
    f_{G,G'}\leq \prod_{\substack{1\leq i\leq c,\\E(U_i)\cap E(G,G')\neq\emptyset}} (L!).
\]
By definition, $|E(G,G')|$ is at least the number of components $U_i$ satisfying $E(U_i)\cap E(G,G')\neq\emptyset$. This completes the proof.
\end{proof}

Given Lemma~\ref{lem:expbound}, our main remaining task is to control the (truncated) exponential moments of $|E(G,G')|$. For technical reasons outlined at the end of this subsection, we will cover $E(G,G')$ by a union $E(G,G')=E_{\for}(G,G')\cup E_{\back}(G,G')$ of two sets which omit lexicographically late and early strings respectively. To ensure that $E(G,G')$ can be covered in this way for $G,G'\in\mathcal S_1$, we add a second restriction to the definition of $\mathcal S_1$ called regularity. This amounts to requiring that both prefixes $[00]$ and $[(k-1)(k-1)]$ appear with roughly the expected frequency among the strings $(s_1,\dots,s_N)$ of $S$.

\begin{definition}
The sequence $S=(s_1,\dots,s_N)\in \mathcal S$ of strings is \textbf{regular} if at most $\left(p_0^2+\frac{p_0p_{k-1}}{2}\right)N$ strings $s_i$ begin with $[00]$ (two consecutive $0$ digits) and at most $\left(p_{k-1}^2+\frac{p_0p_{k-1}}{2}\right)N$ strings begin with $[(k-1)(k-1)]$ (two consecutive $(k-1)$ digits).
\end{definition}

\begin{restatable}{lemma}{Lsparse}\label{lem:Lsparse}
For any $\bp$ and $\varepsilon>0$ there exist $L=L(\bp,\varepsilon)\in\mathbb Z^+$ and $\delta=\delta(\bp,\varepsilon)>0$ such that the following holds. Consider a $\bp$-random sequence $S=(s_1,\dots,s_N)$ of strings of length $K\geq (\widetilde C_{\bp}+\varepsilon)\log(N)$. Then with probability $1-O(N^{-\delta})$, $S$ is regular and $G(S)$ is $L$-sparse.
\end{restatable}

The proof is deferred to the next subsection. $\mathcal S_1$ can now be defined: it consists of the regular sequences $S$ for which $G(S)$ is $L$-sparse for $L=L(\bp,\varepsilon)$ as in Lemma~\ref{lem:Lsparse}. Then Lemma~\ref{lem:Lsparse} exactly states that 
\[ 
    \mu_{\bp}(\mathcal S_0)=O(N^{-\delta})
\] 
for some small $\delta=\delta(\bp,\varepsilon)$.

We remark that the convergence rate $O(N^{-\delta})$ eventually appears as the upper bound for the total variation distance to uniformity (see \eqref{eq:TV} and the next displayed equations in that proof). The rate $O(N^{-\delta})$ seems to be tight in e.g. Proposition~\ref{prop:uppersimple} via Lemma~\ref{lem:UBnumerics}. As a result we use this rate in the statement of Lemma~\ref{lem:Lsparse} although it could be improved. In fact the probability for $S$ to be regular is at least $1-e^{-\Omega_{\bp}(N)}$. The probability for $G(S)$ to be $L$-sparse can be made at most $e^{-CN}$ for any desired $C>0$, if $L=L(C,\bp,\varepsilon)$ is taken sufficiently large.

Next we show how to cover $E(G,G')$ when $S$ and $S'$ are regular.

\begin{definition}

Let $E_{\for}(G)$ consist of all edges $(i,i+1)\in E(G)$ for which the strings $s_i=s_{i+1}$ do \textbf{not} begin with prefix $[(k-1)(k-1)]$. Let $E_{\for}(G,G')=E_{\for}(G)\cap E_{\for}(G')$. Define $E_{\back}(G,G')$ in the same way but with $[(k-1)(k-1)]$ replaced by $[00]$.

\end{definition}

\begin{lemma}\label{lem:forplusback}

If $S,S'\in \mathcal S$ are regular, then \[|E(G,G')|\leq |E_{\for}(G,G')|+|E_{\back}(G,G')|.\] 

\end{lemma}

\begin{proof}

Regularity implies that $E_{\for}(G,G')$ contains all shared edges $(i,i+1)\in E(G,G')$ with 
\[
    i\leq (1-p_{k-1}^2-(p_0p_{k-1}/2))N,
\]
and $E_{\back}(G,G')$ contains all shared edges $(i,i+1)\in E(G,G')$ with 
\[
    i\geq (p_0^2+(p_0p_{k-1}/2))N.
\]
Since 
\[
p_0^2+p_0p_{k-1}+p_{k-1}^2<(p_0+p_{k-1})^2\leq 1
\]
we obtain
\[
(p_0^2+(p_0p_{k-1}/2))N\leq  (1-p_{k-1}^2-(p_0p_{k-1}/2))N.
\]
Therefore
\[
    E_{\for}(G,G')\cup E_{\back}(G,G')=E(G,G')
\]
which implies the result.
\end{proof}

Using symmetry to suppress the identical case of $E_{\back}$, to establish the mixing time upper bound in Theorem~\ref{thm:main} it remains to verify the following lemma.

\begin{restatable}{lemma}{tails}
\label{lem:tails}
For any $\bp$ and positive reals $\varepsilon$ and $t$, there is $\delta=\delta(\bp,\varepsilon,t)$ such that if $K\geq (\overline{C}_{\bp}+\varepsilon)\log(N)$ then
\[ 
    \mathbb E[e^{t\cdot |E_{\for}(G,G')|}]\leq 1+O(N^{-\delta}).
\]
\end{restatable}

Indeed, the mixing time upper bound \eqref{eq:UB} in Theorem~\ref{thm:main} easily follows from the results above as we show now.

\begin{proof}[Proof of \eqref{eq:UB}] Let $\delta>0$ be sufficiently small depending on $(\bp,\varepsilon,L,t)$, some of which are yet to be chosen. By~\eqref{eq:2.1} and~\eqref{eq:2.2},
\begin{align}
\nonumber 
    d_N(K)&=\frac{1}{2}\cdot\mathbb E^{\sigma}\left|\mathbb E^{S}[f_{G(S),\sigma}]-1\right|\\
\label{eq:TV}
    &\leq \frac{1}{2}\cdot\sqrt{\mathbb E^{S,S'}\left[(f_{G,G'}-1)1_{S,S'\in \mathcal S_1}\right]}
    +
    \mu_{\bp}(\mathcal S_0).
\end{align}
(It follows from \eqref{eq:2.2} that the expression inside the square-root is non-negative.) Since $\mu_{\bp}(\mathcal S_0)=O(N^{-\delta})$ by Lemma~\ref{lem:Lsparse}, it remains to estimate $\mathbb E^{S,S'\in\mathcal S}\left[(f_{G,G'} -1)1_{S,S'\in \mathcal S_1}\right]$. Using Lemma~\ref{lem:expbound} in the first step, then Lemma~\ref{lem:forplusback} and finally Lemma~\ref{lem:tails} with $t=2\log(L!)$, we obtain
$$
\begin{WithArrows}
\hspace{-1.5cm}\mathbb E^{S,S'\in\mathcal S}\left[(f_{G,G'} -1)1_{S,S'\in \mathcal S_1}\right]&\leq \mathbb E^{S,S'}[ \left((L!)^{|E(G,G')|}-1\right)1_{S,S'\in \mathcal S_1}]\Arrow{Lemma~\ref{lem:forplusback}}\\
&\leq \mathbb E[\left((L!)^{|E_{\for}(G,G')|+|E_{\back}(G,G')|}-1\right)1_{S,S'\in \mathcal S_1}]\\
&\leq \mathbb E[(L!)^{|E_{\for}(G,G')|+|E_{\back}(G,G')|}-1]\\
&\leq \frac{\mathbb E[(L!)^{2|E_{\for}(G,G')|}]+\mathbb E[(L!)^{2|E_{\back}(G,G')|}]}{2}-1 \Arrow{Lemma~\ref{lem:tails}}\\
&\leq O(N^{-\delta}).
\end{WithArrows}
$$
Combining the above, we conclude that $d_N(K)\leq O(N^{-\delta})$ when $K\geq (\overline{C}_{\bp}+\varepsilon)\log(N)$.
\end{proof}

The above constitutes a complete proof for the upper bound, except that Lemmas~\ref{lem:fbound}, \ref{lem:Lsparse} and \ref{lem:tails} are yet to be proved. The first two are not difficult and are handled in the next subsection. Lemma~\ref{lem:tails} is more challenging and its proof occupies Sections~\ref{sec:proof} and \ref{sec:UBfinish}. We now outline our approach to Lemma~\ref{lem:tails}, which starts from the following basic fact. Suppose $X\in\mathbb N$ is a non-negative integer-valued random variable satisfying the uniform hazard rate bound
\begin{equation}
\label{eq:hazard-rate}
    \sup_{j\geq 0}\mathbb P[X\geq j+1|X\geq j]\leq O(N^{-\delta})
\end{equation}
for some $\delta>0$. Then $X$ is stochastically dominated by a geometric random variable with mean $O(N^{-\delta})$, and therefore $\mathbb E[e^{tX}]=1+O(e^t N^{-\delta})=1+o(1)$ for any constant $t$. To prove Lemma~\ref{lem:tails}, we will implement this idea with $X=|E_{\for}(G,G')|$. We explore $G$ and $G'$ by revealing their strings together in order, so that 
\[
    (s_1,\dots,s_i)
    \quad
    \text{ and }
    \quad
    (s_1',\dots,s_i')
\]
have been revealed at time $i\in [N]$. We show that at \emph{any} time, the expected number of unrevealed edges in $E_{\for}(G,G')$ is at most $O(N^{-\delta})$. That is, almost surely,
\begin{equation}
\label{eq:cond-bound}
    \mathbb E\left[\big|E_{\for}(G,G')|_{\{i,i+1,\dots,N\}}\big| ~ \Big| ~ (s_1,\dots,s_i,s_1',\dots,s_i') \right]\leq O(N^{-\delta}).
\end{equation}
(Here we write $E_{\for}(G,G')|_{\{i,i+1,\dots,N\}}$ to indicate the set of edges $(j,j+1)\in E_{\for}(G,G')$ with $j\geq i$.)
The estimate \eqref{eq:cond-bound} readily implies Lemma~\ref{lem:tails} analogously to the above discussion on \eqref{eq:hazard-rate}. See Lemma~\ref{lem:geodom} for a detailed proof. 

As a first step towards establishing \eqref{eq:cond-bound}, in Section~\ref{sec:proof} we prove for $K\geq (\overline{C}_{\bp}+\varepsilon)\log(N)$ the weaker first moment bound 
\begin{equation}\label{eq:1stmoment}
\E\left[|E(G,G')|\right]\leq O(N^{-\delta}).
\end{equation}
In Section~\ref{sec:UBfinish} we use \eqref{eq:1stmoment} to show \eqref{eq:cond-bound}. The idea is to group the set of possible future strings
\[
    \{s\in [k]_0^K:s\geq_{\lex} s_i\}
\] 
into a small number of blocks. Here each block $B_x$ consists of all strings beginning with some prefix $x\in [k]_0^M$ (where $M=M(x)$ depends on $x$).  Such a block $B_x$ is essentially equivalent to a copy of $[k]_0^{K-M}$. The idea is to first estimate the left-hand side of \eqref{eq:cond-bound} by a sum over blocks (using Cauchy--Schwarz several times), and to then estimate the contribution of each block using \eqref{eq:1stmoment}. The total number of blocks will always be $O(\log N)\leq N^{o(1)}$. Therefore summing over blocks is no problem (up to adjusting the value of $\delta$ slightly) as long as the hypothesis of \eqref{eq:1stmoment} applies ``within'' each block.

To illustrate the key reason for introducing $E_{\for}$, let us explain why \eqref{eq:cond-bound} can be false if $E_{\for}(G,G')$ is replaced by $E(G,G')$. Suppose that $s_i=s_i'=[(k-1)^K]$ holds for some $i \in [N]$. Then conditioning on $(s_i,s_i')$ forces $s_j=s_j'=[(k-1)^K]$ for all $j>i$. Hence $E(G,G')$ must contain all the edges $(i,i+1),(i+1,i+2),\dots,(N-1,N)$ and so
\[
    \mathbb E\left[\big|E(G,G')|_{\{i,i+1,\dots,N\}}\big| ~ \Big| ~ (s_1,\dots,s_i,s_1',\dots,s_i') \right] = N-i+1.
\]
However working with $E_{\for}(G,G')$ prevents such situations by halting exploration once either $s_i$ or $s_i'$ becomes too lexicographically late. This circumvents the above obstruction because the left-hand side of \eqref{eq:cond-bound} is trivially $0$ unless a lot of ``space'' in $[k]_0^K$ remains available for future strings $(s_{i+1},\dots,s_N,s_{i+1}',\dots,s_N')$.

In fact, this guaranteed available space is also directly helpful in implementing the block decomposition strategy outlined above. Namely for any prefix $x\in [k]_0^M$, it ensures that the distribution for the number of strings $(s_{i+1},\dots,s_N)$ starting with $x$ cannot increase too much from conditioning on $(s_1,\dots,s_i)$ (see Lemma~\ref{lem:conbin} for a precise statement). This is important because when applying \eqref{eq:1stmoment} to the block of strings starting with some prefix $x$, we replace $N$ by the number of strings starting with $x$ (and also replace $K$ by $K-M$). In short, we must ensure that the hypothesis of \eqref{eq:1stmoment} holds within each block.

\subsection{Proof of Lemmas~\ref{lem:fbound} and \ref{lem:Lsparse}}\label{subsec:lemmas}

Here we prove Lemmas~\ref{lem:fbound} and \ref{lem:Lsparse}, thus reducing the proof of the mixing time upper bound \eqref{eq:UB} to establishing Lemma~\ref{lem:tails}.

\begin{proof}[Proof of Lemma~\ref{lem:fbound}]

Let $(v_1,\dots,v_a)$ be the vertex-sizes of the $G$-components and $(w_1,\dots,w_{b})$ be the vertex-sizes of the $G'$-components.

We first claim that
\begin{equation}
\label{eq:fgg}
    f_{G,G'}=\frac{\big(\prod_{i=1}^a v_i!\big)\cdot\big(\prod_{j=1}^b w_j!\big)}{\prod_{\ell=1}^c u_{\ell}!}.
\end{equation}
Indeed this follows by writing
\begin{align*}
    f_{G,G'}
    &=
    \mathbb E^{\sigma}[f_{G,\sigma}f_{G',\sigma}]\\
    &=
    \mathbb E^{\sigma}\left[1_{\sigma^G=\sigma}\cdot 1_{\sigma^{G'}=\sigma}\cdot \left(\prod_{i=1}^a v_i!\right)\cdot\left(\prod_{j=1}^b w_j!\right)\right]\\
    &=
    \mathbb E^{\sigma}\left[1_{\sigma^U=\sigma}\right]\cdot \left(\prod_{i=1}^a v_i!\right)\cdot\left(\prod_{j=1}^b w_j!\right)\\
    &=
    \frac{\big(\prod_{i=1}^a v_i!\big)\cdot\big(\prod_{j=1}^b w_j!\big)}{\prod_{{\ell}=1}^c u_{\ell}!}.
\end{align*}
Decomposing the product in \eqref{eq:fgg} based on the components $U_i\in\mathcal C(U)$ implies
\begin{equation}
\label{eq:fprod}f_{G,G'}=\prod_{\ell} f_{G,G',U_{\ell}}\end{equation}
where we define
\[
    f_{G,G',U_{\ell}}
    \equiv
    \frac{\left(\prod_{i:G_i\subseteq U_{\ell}} v_i!\right)\cdot\left(\prod_{j:G'_j\subseteq U_{\ell}} w_j!\right)}{u_{\ell}!}.
\]
Observe that in general, for any positive integers $m_1,\dots,m_n,M$ with 
\[ 
    \sum_{i=1}^n (m_i-1)\leq M-1,
\]
one has $\prod_{i=1}^n m_i!\leq M!$. Indeed both sides can be written as a product of at most $M-1$ integers at least $2$, and the $M-1$ numbers appearing in the product for $M!$ are clearly larger. In particular, this holds for $M=u_{\ell}$ whenever $m_1,\dots,m_n$ are the vertex-sizes of edge-disjoint subinterval graphs of $V(U_{\ell})$. It directly implies
\begin{align*}
    \prod_{i:G_i\subseteq U_{\ell}} v_i!&\leq u_{\ell}!,\\
    \prod_{j:G'_j\subseteq U_{\ell}} w_j!&\leq u_{\ell}!
\end{align*}
from which it follows that $f_{G,G',U_{\ell}}\leq u_{\ell}!$ holds. Moreover if $U_{\ell}$ does not contain any edge in $E(G,G')$ then the $G$-components and $G'$-components are collectively edge-disjoint. Hence for such $U_{\ell}$,
\[
    \left(\prod_{i:G_i\subseteq U_{\ell}} v_i!\right)\cdot \left(\prod_{j:G'_j\subseteq U_{\ell}} w_j!\right)\leq u_{\ell}!
\]
which implies $f_{G,G',U_{\ell}}\leq 1$. Substituting these estimates into \eqref{eq:fprod} implies \eqref{eq:fggbound}.
\end{proof}

The next lemma is used to show Lemma~\ref{lem:Lsparse}.

\begin{lemma}
\label{lem:connorepeat}
For $K\geq (\widetilde C_{\bp}+\varepsilon)\log(N)$, there is $\delta(\bp,\varepsilon)>0$ so that the following holds. Conditioned on any initial strings $s_1,s_2,\dots,s_i$, none of which begin with $[(k-1)(k-1)]$, the conditional probability that $s_i=s_{i+1}$ is at most $O(N^{-\delta})$. 

\end{lemma}

\begin{proof}
Recall the definitions \eqref{eq:tx}, \eqref{eq:lambdax}, \eqref{eq:Jx} and the subsequent discussion. We use the sampling model of $N$ i.i.d.-then-sorted uniform random variables, letting 
\[
    0\leq a_1\leq a_2\leq \dots\leq a_N\leq 1
\]
be uniformly random before being sorted and then choosing $s_j$ such that $a_j\in J_{s_j}$ for each $1\leq j\leq N$. 

Recall that we condition on $s_i$. Let us now condition further on the value $a_i\in J_{s_i}$. Then the remaining numbers $a_j$ for $j>i$ are, up to sorting, conditionally i.i.d. in $[a_i,1]$. The crucial observation is that the interval $[a_i,1]$ has length lower bounded by $1-a_i\geq p_{k-1}^2\geq p_{\Min}^2$. Indeed, $a_i< 1-p_{k-1}^2$ is equivalent to the assumption that $s_i$ does not begin with $[(k-1)(k-1)]$. (For instance, note that $J_{[(k-1)(k-1)]}=[1-p_{k-1}^2,1)$.) Meanwhile the length of $J_{s_i}$ is $\lambda_{s_i}$.

Combining these observations, it follows that the conditional distribution for the number of $j>i$ with $s_j=s_i$ is stochastically dominated by a $\Bin(N,p_{\Min}^{-2}\lambda_{s_i})$ random variable, regardless of the value $a_i$. Averaging over the unknown $a_i$, the same stochastic domination holds conditioned on just $(s_1,\dots,s_i)$. Since $K\geq (\widetilde C_{\bp}+\varepsilon)\log(N)$ was assumed,
\[
    \lambda_{s_i}\leq (p_{\M})^K\leq N^{-1-\delta}.
\]
The result now follows.
\end{proof}

\begin{proof}[Proof of Lemma~\ref{lem:Lsparse}]

The regularity readily follows from Chernoff estimates so we focus only on the $L$-sparsity. First, Lemma~\ref{lem:connorepeat} implies that $\mathbb P[s_{i+1}=s_i|(s_1,\dots,s_i)]\leq O(N^{-\delta})$ whenever $s_i<_{\lex}[(k-1)(k-1)]$. A simple Markovian coupling now implies that the set of edges formed by strings $s_i<_{\lex}[(k-1)(k-1)]$ is stochastically dominated by instead choosing each edge independently with probability $O(N^{-\delta})$. By symmetry the same holds for edges formed by strings starting with $[(k-1)(k-1)]$. Call these ordinary edges and final edges, respectively. 

A simple Chernoff bound implies that for $L\geq 1000\delta^{-1}$, each interval $\{i,i+1,\dots,i+L-1\}$ of $L$ consecutive vertices contains at most $L/6$ ordinary edges and at most $L/6$ final edges with probability $1-O_L\left(\frac{1}{N^2}\right)$. Since $L/6+L/6=L/3$, union bounding over at most $N$ such length-$L$ intervals shows that $L$-sparsity holds with probability at least $1-O(N^{-1})\geq 1-O(N^{-\delta})$.
\end{proof}

\section{Upper Bounding the Expected Shared Edges}\label{sec:proof}

Define the constant
\[
    \underline{C}_{\bp}\equiv\max\left(C_{\bp},\frac{1}{\log(1/p_0)},\frac{1}{\log(1/p_{k-1})}\right)\leq \overline{C}_{\bp}
.
\]
The purpose of this section is to prove the following crucial result. 

\begin{restatable}{proposition}{uppersimple}\label{prop:uppersimple}
For any $\varepsilon>0$, if $K\geq \left(\underline{C}_{\bp}+\varepsilon\right)\log(N)$ holds then \[
\E\left[|E(G,G')|\right]\leq O(N^{-\Omega_{\bp}(\varepsilon)}).
\]
\end{restatable}

We eventually need to control the (truncated) \emph{exponential} moments of $E(G,G')$. However Proposition~\ref{prop:uppersimple} is the most involved part of upper-bounding the mixing time, and the value $C_{\bp}=\frac{3+\theta_{\bp}}{4\psi_{\bp}(2)}$ emerges naturally in its proof. We note that for our main goal of establishing cutoff, proving Proposition~\ref{prop:uppersimple} only for $K\geq \left(\overline{C}_{\bp}+\varepsilon\right)\log(N)$ would suffice just as well. However there is no additional difficulty in proving Proposition~\ref{prop:uppersimple} as stated. Moreover the case $\underline{C}_{\bp}\neq \overline{C}_{\bp}$ amounts to an interesting discrepancy between the first moment and exponential moment behavior of $|E(G,G')|$. See Remark~\ref{rem:gap} for more discussion of this point.

Let us mention that after further preparation in Subsection~\ref{subsec:notation}, we provide in Subsection~\ref{subsec:outline} a proof outline for Proposition~\ref{prop:uppersimple}.

\subsection{Preparation for the Upper Bound Proof}\label{subsec:notation}

We now introduce several more technical definitions. As a convention, $\bp$ and $\varepsilon$ will be treated as fixed, while $\delta=\delta(\bp,\varepsilon)$ will be taken sufficiently small. As before $G$ and $G'$ will always be independent $\bp$-random shuffle graphs. Moreover $s$ will denote strings of length $K$ while $x$ will denote strings of arbitrary length $M\leq K$.

\subsubsection{Lexicographic Subintervals and Blocks}

\begin{figure}[H]
\centering
\begin{framed}
\includegraphics[width=\linewidth]{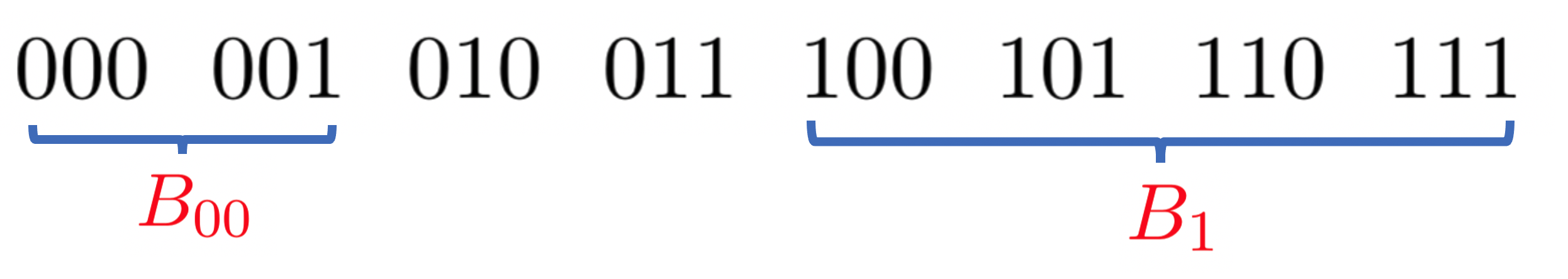}
\caption{The blocks $B_{00}$ and $B_1$ are shown for $k=2$ and $K=3$.}
\end{framed}
\label{fig:J}
\end{figure}

For a string $x$ of length $M$, define its \emph{block} $B_{x}\subseteq [k]_0^K$ to be the set of strings of length $K$ beginning with $x$. Hence $B_x$ consists of $k^{K-M}$ strings. Given a lexicographically sorted sequence $(s_1,\dots,s_N)\in \mathcal S$ of strings, define the discrete interval $\mathcal I(B_x)\subseteq [N]$ by
 \[
 \mathcal I(B_x)\equiv\{i\in [N]:s_i\in B_x\}=\{\iota(x),\iota(x)+1,\dots,\tau(x)\}.
 \] 
In general, we define 
\[\iota(x)=|\{i\in [N]:s_i<_{\lex}x\}|+1,\quad\quad\tau(x)=|\{i\in [N]:s_i<_{\lex}x\text{ or } s_i\in B_x\}|.\] 
This ensures $|\mathcal I(B_x)|=\tau(x)-\iota(x)+1$ even if $\mathcal I(B_x)$ is empty. Observe that for fixed $x$ (recall the definitions \eqref{eq:tx} and \eqref{eq:lambdax}),
\begin{align}
    \label{eq:IBx}|\mathcal I(B_x)|&\sim \Bin(N,\lambda_x),\\
    \label{eq:iota}\iota(x)&\sim \Bin(N,t_x)+1,\\
    \label{eq:tau}\tau(x)&\sim \Bin(N,t_x+\lambda_x).
\end{align}

Finally define $G_{B_x}$ to be the induced subgraph of $G$ with vertex set $\mathcal I(B_x)$, which retains the edges $(i,i+1)\in E(G)$ such that $s_i=s_{i+1}\in B_x$. Denote its edge set by $E(G_{B_x})$.

\subsubsection{Entropy}

We will require the entropy function. Given any $k$-tuple of non-negative real numbers $(a_0,\dots,a_{k-1})$ with sum $a_{\tot}$, let
\[
H(a_0,\dots,a_{k-1})=\frac{\sum_{i=0}^{k-1} a_i\log\left(\frac{a_{\tot}}{a_i}\right)}{a_{\tot}}\] be the entropy of the discrete probability distribution with weights $\left(a_i/a_{\tot}\right)_{i=0}^{k-1}.$ If $a_0=\dots=a_{k-1}=0$ then set $H(a_0,\dots,a_{k-1})=0$. The following result allows approximation of multinomial coefficients using entropy. (The values $a_i\log(N)$ correspond to the normalization in Definition~\ref{def:digitprofile} just below.)

\begin{proposition}[{\cite[Lemma 2.2]{csiszar2004information}}]
\label{prop:entropy}
For fixed $A\geq 0$ and any non-negative real numbers $a_0,\dots,a_{k-1}\in [0,A]$ satisfying $a_i\log(N)\in\mathbb Z$,
\[
    N^{a_{\tot} H(a_0,\dots,a_{k-1})-o_N(1)}\leq \binom{a_{\tot}\log(N)}{a_0\log(N),\dots,a_{k-1}\log(N)}\leq N^{a_{\tot} H(a_0,\dots,a_{k-1})}.
\] 
Here the term $o_N(1)$ tends to $0$ for any fixed $A$ as $N\to\infty$, uniformly in the values $a_0,\dots,a_{k-1}\in [0,A]$.
\end{proposition}

The following special definitions will also be convenient. For $t>0$, let $\bp^t$ be the probability distribution on $[k]_0$ given by $(\bp^t)_i=\frac{p_i^t}{\phi_{\bp}(t)}$. Define
\begin{equation}
\label{I-def}
    I(\bp,\bp^{t})\equiv D_{\term{KL}}(\bp^t\mid\mid\bp)+H(\bp^t)=\sum_{i=0}^{k-1} (\bp^{t})_i \log(1/p_i)=\sum_{i=0}^{k-1} \frac{p_i^t \log(1/p_i)}{\phi_{\bp}(t)}>0.
\end{equation}
It is not difficult to verify the identity
\begin{equation}\label{eq:IDI}
H(\bp^t)=t\cdot I(\bp,\bp^t)-\psi_{\bp}(t), \quad t\in \mathbb R^+.
\end{equation}

\subsubsection{Digit Profile}

\label{subsubsec:digit-profile}

Here we define several notions based on the \emph{digit profile} of a string, which tracks how many of each digit a string contains, as well as initial digits of $0$ or $k-1$.

\begin{definition}
\label{def:digitprofile}
For a string $x\in [k]_0^M$, the \textbf{digit profile} of $x$ is the $(k+2)$-tuple 
\[
(b_0(x),b_{k-1}(x),c_0(x),\dots,c_{k-1}(x))\in\left(\mathbb Z/\log N\right)^{k+2}
\] 
of non-negative real numbers summing to $b_0+b_{k-1}+\sum_{i=0}^{k-1} c_i=\frac{M}{\log(N)}$ defined as follows. $b_0\log(N)$ is the number of initial $0$-digits in $x$ and $b_{k-1}\log(N)$ is the number of initial $(k-1)$-digits (so $\min(b_0,b_{k-1})=0$). After the first $(b_0+b_{k-1})\log(N)$ digits, $x$ contains $c_i\log(N)$ digits of $i$ for each $i\in [k]_0$. 
\end{definition}

The normalization $\frac{1}{\log N}$ above is taken so that the total sum $\frac{M}{\log N}$ is of constant order. We next define constants depending on the digit profile of $x$. Let \[c_{\tot}(x)=\sum_{i=0}^{k-1} c_i(x)\] be the number of digits in $x$ after the initial $0$ or initial $(k-1)$ digits. Also define
\begin{align*}
\cL(x)&\equiv 1-b_0\log\left(\frac{1}{p_0}\right)-b_{k-1}\log\left(\frac{1}{p_{k-1}}\right)-\sum_{i=0}^{k-1} c_i\log\left(\frac{1}{p_i}\right)=1+\log_N(\lambda_x),\\
\cP(x)&\equiv\frac{1-b_0\log\left(\frac{1}{p_0}\right)-b_{k-1}\log\left(\frac{1}{p_{k-1}}\right)}{2},\\
\cD(x)&\equiv \cL(x)-\cP(x)=\frac{1-b_0\log\left(\frac{1}{p_0}\right)-b_{k-1}\log\left(\frac{1}{p_{k-1}}\right)}{2}-\sum_{i=0}^{k-1} c_i\log\left(\frac{1}{p_i}\right),\\
\cW(x)&\equiv\left(\frac{M-K}{\log N}\right)\psi_{\bp}(2)=\left(b_0+b_{k-1}+c_{\tot}-\frac{K}{\log N}\right)\psi_{\bp}(2)<0,\\
\cE(x)&\equiv c_{\tot} H(c_0,\dots,c_{k-1})+5\cL-2\cP+2\cW. 
\end{align*}
Finally say $x$ is \emph{$\delta$-stable} if
\begin{equation}
\label{eq:delta-stable}
    \cL(x)-\cP(x) \in [\delta,2\delta].
\end{equation}
The typical size of $|\mathcal I(B_x)|$ is $N^{\cL}$ while $N^{\cP}$ is the order of fluctuations for $\iota(x)$ and $\tau(x)$. $\cW$ is related to the typical number of $G$-edges coming from strings in $B_x$. $\cE$ is related to the typical number of $G$-edges coming from strings of the same digit profile as $x$. Note that when $b_0=b_{k-1}=0$ we have $\cP=\frac{1}{2}$. As explained in the next subsection, this corresponds to $\iota(x)$ and $\tau(x)$ having fluctuations of order $N^{1/2}$.

\subsection{Proof Outline for Proposition~\ref{prop:uppersimple}}
\label{subsec:outline}

We now outline the proof of Proposition~\ref{prop:uppersimple}. Except for the end of this outline we will only consider strings $x$ with $b_0(x)=b_{k-1}(x)=0$ so that the interval $J_x\in [0,1]$ is a constant distance from the boundary points $\{0,1\}$. We will take $\delta\ll\varepsilon$ to be a small constant, and simply write $\delta$ when a constant multiple such as $4\delta$ would be technically correct. Since we are targeting an upper bound $O\left(N^{-\Omega_{\bp}(\varepsilon)}\right)$ in Proposition~\ref{prop:uppersimple}, factors of $N^{O(\delta)}$ can usually be thought of as small. 

 The first idea is to start from the empty block $B_{\emptyset}=[k]_0^K$ and recursively refine the partition of $[k]_0^K$ by decomposing a block $B_x$ into $k$ smaller blocks via
\[
B_x=\bigcup_{i\in [k]_0} B_{xi}.
\]
For example when $k=2$ such a refinement might proceed as
\[
B_{\emptyset}\to B_0\cup B_1\to B_{00}\cup B_{01}\cup B_1=[2]_0^K.
\]
We recursively refine the partition $B_{\emptyset}$ in this way until each block $B_x$ in the partition has size $\mu_{\bp}(B_x)\approx N^{-\frac{1}{2}+\delta}$; this is formally carried out in Lemma~\ref{lem:Lstable}. The set of strings $x$ used in the resulting partition is denoted by $\mathcal L_{\stable}$, so that we obtain
\begin{equation}\label{eq:partition-preview}
[k]_0^K=\bigcup_{x\in \mathcal L_{\stable}} B_x,\quad\text{ and }\quad [N]=\bigcup_{x\in \mathcal L_{\stable}} \mathcal I(B_x).
\end{equation}
as in Lemma~\ref{lem:partition}. The first and last indices $\iota(x)$ and $\tau(x)$ of $\mathcal I(B_x)$ are (non-independent) binomial random variables with $N$ trials, hence each fluctuate by at most $O(N^{1/2})$ with high probability.

The upshot of the above is that the random set $\mathcal I(B_x)$ agrees with a discrete deterministic interval of size $|NJ_x\cap\mathbb Z|\approx N^{\frac{1}{2}+\delta}$ up to boundary fluctuations $|\iota(x)-Nt_x|$ and $|\tau(x)-N(t_x+\lambda_x)|$ which are smaller than $N^{\frac{1}{2}+\delta}$ with high probability.
Because the random interval $\mathcal I(B_x)$ has typical size of larger order than the fluctuations of its left and right endpoints, we may think of $\mathcal I(B_x)$ as being nearly deterministic. In line with this intuition, we show in Lemma~\ref{lem:blocksvalid} that given any $i\in [N]$ there exist adjacent $x_{i,1},x_{i,2}\in\mathcal L_{\stable}$ such that $i\in \mathcal I(B_{x_{i,1}})\cup\mathcal I(B_{x_{i,2}})$ holds with extremely high probability. Combining this with AM-GM, we show in Lemma~\ref{lem:AMGM} that $\E[|E(G,G')|]$ is upper bounded by the expected number of shared edges from pairs $(G_{B_x},G'_{B_x})$ of matching blocks as follows.
\begin{align}
\label{eq:UBintuitive}
\mathbb E[|E(G,G')|]&\lesssim \sum_{x\in\mathcal L_{\stable}}\mathbb E[E(G_{B_x},G'_{B_x})]\\
\nonumber&= \sum_{x\in\mathcal L_{\stable}}\sum_{i=1}^{N-1} \mathbb P[(i,i+1)\in E(G_{B_x})]^2.
\end{align}
Here the informal notation $\lesssim$ hides a constant factor and a negligible additive term. 

Our next objective is to upper-bound for each $x$ the probability $\mathbb P[(i,i+1)\in E(G_{B_x})]$ appearing in \eqref{eq:UBintuitive}. We do this by conditioning on the multiset $S_x$ of strings appearing in $\mathcal I(B_x)$ and averaging over the still-random external strings. Although this conditioning determines the size and internal edge-structure of $\mathcal I(B_x)$, the \emph{position} of $\mathcal I(B_x)$ is conditionally random. Indeed the position of the interval $\mathcal I(B_x)$ depends on the number of external strings lexicographically smaller than $x$, which we have not conditioned on. This shift is conditionally binomial with order $N^{1/2}$ fluctuations. Crucially, these fluctuations ``homogenize'' the edge locations within each block $B_x$. Indeed averaging over these external shifts, it follows that
\begin{equation}\label{eq:homogenize}
\max_{i\in [N-1]}\mathbb P[(i,i+1)\in E(G_{B_x})|S_x]\lesssim \frac{|E(G_{B_x})|}{N^{1/2}}.
\end{equation}
It is not difficult to control the typical size $|E(G_{B_x})|$. Moreover since the location of $\mathcal I(B_x)$ is almost deterministic, the above probability is negligibly small for all but $O(\mathbb E[|\mathcal I(B_x)|])=O(N^{\frac{1}{2}+\delta})$ values of $i$. Combining these considerations leads to an upper bound on 
\begin{equation}\label{eq:x-sum}
    \sum_{i=1}^{N-1} \mathbb P[(i,i+1)\in E(G_{B_x})]^2.
\end{equation}
The precise bound requires some care to state and is given in Lemma~\ref{lem:blockbound}.

The preceding argument allowed us to estimate \eqref{eq:x-sum} for each $x$. In light of \eqref{eq:UBintuitive}, it remains to sum over $x\in\mathcal L_{\stable}$. The last key point is that all $x\in\mathcal L_{\stable}$ with a given digit profile contribute essentially identically. Moreover there are only $\log(N)^{O(1)}\leq N^{o(1)}$ possible digit profiles. It therefore suffices to count the number of $x\in\mathcal L_{\stable}$ with each digit profile and then determine the maximum total contribution of any fixed digit profile. This count is easily approximated using Proposition~\ref{prop:entropy}. The resulting maximum turns out to be achieved when $x$ has digit frequencies approximately given by $\bp^{\theta_{\bp}}$, which leads to the appearance of the constant $C_{\bp}$.

So far, this outline considered only blocks $B_x$ with $b_0(x)=b_{k-1}(x)=0$. When $b_0(x)$ or $b_{k-1}(x)$ is large the fluctuations of $\iota(x)$ and $\tau(x)$ shrink, simply because the variance $Np(1-p)$ of a $\Bin(N,p)$ random variable shrinks when $p$ is close to $0$ or $1$. To handle such prefixes $x$ requires a slightly revised definition of $\mathcal L_{\stable}$. In general the fluctuations of $\iota(x)$ and $\tau(x)$ should be slightly smaller than the typical size of $\mathcal I(B_x)$; this is precisely the definition of $\delta$-stability in \eqref{eq:delta-stable}. It turns out that the resulting maximization problem over digit profiles nearly reduces to considering those with $b_0=b_{k-1}=0$. Indeed by an elementary linearity argument (see \eqref{eq:linear-path}), the only other digit profiles that must be considered are the degenerate cases with $c_0=c_1=\dots=c_{k-1}=0$ in which $x$ consists of all $0$ digits or all $(k-1)$ digits. These cases are much simpler and lead to the requirement that 
\[
    \underline{C}_{\bp}\geq \max\left(\frac{1}{\log(1/p_0)},\frac{1}{\log(1/p_{k-1})}\right).
\]
During a first reading of the next subsection it may be easier to focus on the main case $b_0=b_{k-1}=0$ so that the proofs match the outline above more closely. 

Finally, we remark that the estimates outlined after \eqref{eq:UBintuitive} lead to the inequality
\[
\mathbb E[|E(G,G')|]\lesssim N^{O(\delta)} \sum_{x\in\mathcal L_{\stable}}\frac{\mathbb E[|E(G_{B_x})|]^2}{\mathbb E[|\mathcal I(B_x)|]}.
\]
Hence for the purpose of counting edges in $E(G,G')$, each block $B_x$ behaves approximately like an i.i.d. point process of edges in $\mathcal I(B_x)$ with $x$-dependent edge probability $\frac{\mathbb E[|E(G_{B_x})|]}{\mathbb E[|\mathcal I(B_x)|]}$. In fact \eqref{eq:homogenize} states that this holds more precisely at the level of individual edge probabilities. These hold precisely because the boundary fluctuations of $\mathcal I(B_x)$ are only slightly smaller than $\E[|\mathcal I(B_x)|]$, so that the homogenizing effect of the random shifts is near-total. Somewhat fancifully, one might then view the partition \eqref{eq:partition-preview} as analogous to an ergodic or pure state decomposition.

\subsection{The Partition into Stable Blocks}

We now turn to a tree-based partition of $[k]_0^K$ into blocks $B_x$. Define the $k$-ary rooted tree $\mathcal T=\mathcal T_{k,K}$ of depth $K$ which consists of all $[k]_0$-strings of length $M$ at each level $0\leq M\leq K$. Here the children of $s\in [k]_0^M$ are the concatenations $s0,s1,\dots,s(k-1) \in [k]_0^{M+1}$. Hence the leaves of $\mathcal T$ are given by $[k]_0^K$ while the root of $\mathcal T$ is the empty string $\emptyset$. Recall from Subsection~\ref{subsec:notation} the definition
\[
    \cD(x)=\cL(x)-\cP(x)=\frac{1-b_0\log\left(\frac{1}{p_0}\right)-b_{k-1}\log\left(\frac{1}{p_{k-1}}\right)}{2}-\sum_{i=0}^{k-1} c_i\log\left(\frac{1}{p_i}\right).
\]
Moreover, as in Proposition~\ref{prop:uppersimple}, we assume throughout that $K\geq (\underline{C}_{\bp}+\varepsilon)\log(N)$ holds for some small, fixed $\varepsilon>0$.

\begin{lemma}
\label{lem:gfunc}
Let $x$ be the parent of $y$ in $\mathcal T$. Then
\begin{align}\label{eq:xy}
0\leq \cD(x)-\cD(y)\leq O\left(\frac{1}{\log(N)}\right).
\end{align}
Moreover $\cD$ takes value $\cD(\emptyset)=\frac{1}{2}$ at the root, while $\cD(s)\leq -\Omega_{\bp}(\varepsilon)$ for any leaf $s\in [k]_0^K$. 

\end{lemma}

\begin{proof}

The values $b_0,b_{k-1},c_0,\dots,c_{k-1}$ each change by $O\left(1/\log(N)\right)$ between neighboring vertices in $\mathcal T$, which shows that 
\[
|\cD(x)-\cD(y)|\leq O\left(\frac{1}{\log(N)}\right).
\] 
Moreover since $\cD$ is decreasing in each coordinate of the digit profile it follows that $\cD(x)-\cD(y)\geq 0$. This concludes the proof of \eqref{eq:xy}.

When $x=\emptyset$ is the root, $b_0=b_{k-1}=c_0=\dots=c_{k-1}=0$, and so $\cD(\emptyset)=\frac{1}{2}$. Finally for any leaf $s\in [k]_0^K$ of $\mathcal T$ we have 
\[
    b_0(s)+b_{k-1}(s)+\sum_{i=0}^{k-1} c_i(s)=K\geq \underline{C}_{\bp}+\varepsilon.
\]
Since $t\to \log(\frac{1}{t})$ is decreasing and positive for $t\in (0,1)$,

\begin{align*}
\cL(s)-\cP(s)&=\frac{1}{2}-b_0\cdot \frac{\log\left(\frac{1}{p_0}\right)}{2}-b_{k-1}\cdot \frac{\log\left(\frac{1}{p_{k-1}}\right)}{2}-\sum_{i=0}^{k-1} c_i\log\left(\frac{1}{p_i}\right)\\
&\leq \frac{1}{2}-\frac{(\underline{C}_{\bp}+\varepsilon)\min(\log(1/p_0),\log(1/p_{k-1}),2\log(1/p_{\M}))}{2}
\end{align*}
By definition $\underline{C}_{\bp} \log(1/p_0)\geq 1$ and $\underline{C}_{\bp} \log(1/p_{k-1})\geq 1$. Moreover Proposition~\ref{prop:imp} implies
\[
2\underline{C}_{\bp}\log(1/p_{\M})\geq\frac{2C_{\bp}}{\widetilde{C}_{\bp}}\geq 1.
\]
Combining yields
\[
\underline{C}_{\bp}\cdot\min(\log(1/p_0),\log(1/p_{k-1}),2\log(1/p_{\M}))\geq 1
\]
which implies the result.
\end{proof}

Define the subtree $\mathcal T_{\stable}\subseteq \mathcal T$ to consist of all $x\in\mathcal T$ with $\cD(x)\geq 2\delta$, as well as all children of such $x$. Let $\mathcal L_{\stable}$ denote the set of leaves of $\mathcal T_{\stable}$. We say a finite rooted tree is a \emph{full $k$-ary tree} if all of its vertices have either $0$ or $k$ children.

\begin{lemma}\label{lem:Lstable}

$\mathcal T_{\stable}$ is a full $k$-ary tree. Moreover $\mathcal L_{\stable}$ consists entirely of $\delta$-stable strings. Finally all $x\in \mathcal L_{\stable}$ are strings of length in $[\Omega_{\bp,\delta}(\log N),K-\Omega_{\bp,\delta}(\log N)]$ and satisfy 
\[
    \cP(x)\geq \delta\quad\text{ and }\quad\cL(x)\geq 2\delta.
\]
\end{lemma}

\begin{proof}

First we explain why $\mathcal T_{\stable}$ is a full $k$-ary tree. The point is that since $\cD(x)$ is decreasing down $\mathcal T$ by Lemma~\ref{lem:gfunc}, the set of strings $x$ with $\cD(x)\geq 2\delta$ forms a subtree, and adding all children of such $x$ therefore yields a full $k$-ary subtree.

Next, Lemma~\ref{lem:gfunc} shows $\cD(\emptyset)=\frac{1}{2}$ while $\cD(s)\leq -\Omega_{\bp}(\varepsilon)$ for any $s$ of length $K$, and also shows $\cD$ has Lipschitz constant $O\left(\frac{1}{\log(N)}\right)$ on $\mathcal T$. It follows that $\mathcal T_{\stable}$ contains all of the first $\Omega(\log(N))$ levels of $\mathcal T$ but none of the last $\Omega(\log(N))$. As a result all $x\in\mathcal L_{\stable}$ have length in 
\[
    [\Omega_{\bp,\delta}(\log(N)),K-\Omega_{\bp,\delta}(\log(N))].
\] 
The fact that all leaves are $\delta$-stable holds because children were added in the definition of $\mathcal T_{\stable}$. Indeed this definition combined with \eqref{eq:xy} implies that
\[
    \cD(x)\in [2\delta-O(1/\log N),2\delta]
\]
for all $x\in\mathcal L_{\stable}$. Moreover recalling the definition \eqref{eq:delta-stable}, all $x\in \mathcal L_{\stable}$ satisfy
\[
    \cP(x)+\delta\leq \cL(x).
\]
Finally the inequality $\cL(x)\leq 2\cP(x)$ holds for any string $x$. Altogether, these inequalities imply $\cP(x)\geq \delta$ and therefore 
\[
    \cL(x)\geq \cP(x)+\delta\geq 2\delta.
\]
\end{proof}

\begin{lemma}\label{lem:partition}

The following partitions (i.e. disjoint unions) hold: \begin{equation}
\label{eq:partition}
[k]_0^K=\bigcup_{x\in \mathcal L_{\stable}} B_x\quad\text{ and }\quad [N]=\bigcup_{x\in \mathcal L_{\stable}} \mathcal I(B_x).
\end{equation}

\end{lemma}

\begin{proof}

The first partition clearly implies the second. The first partition holds because $\mathcal L_{\stable}$ consists of the leaves of $\mathcal T_{\stable}$ and $\mathcal T_{\stable}\subseteq\mathcal T$ is a full $k$-ary subtree. Indeed, it simply asserts that the subtrees of $\mathcal T$ rooted at each $x\in\mathcal L_{\stable}$ partition the leaves of $\mathcal T$.
\end{proof}

\begin{figure}[H]

\centering
\begin{framed}
\includegraphics[width=\linewidth]{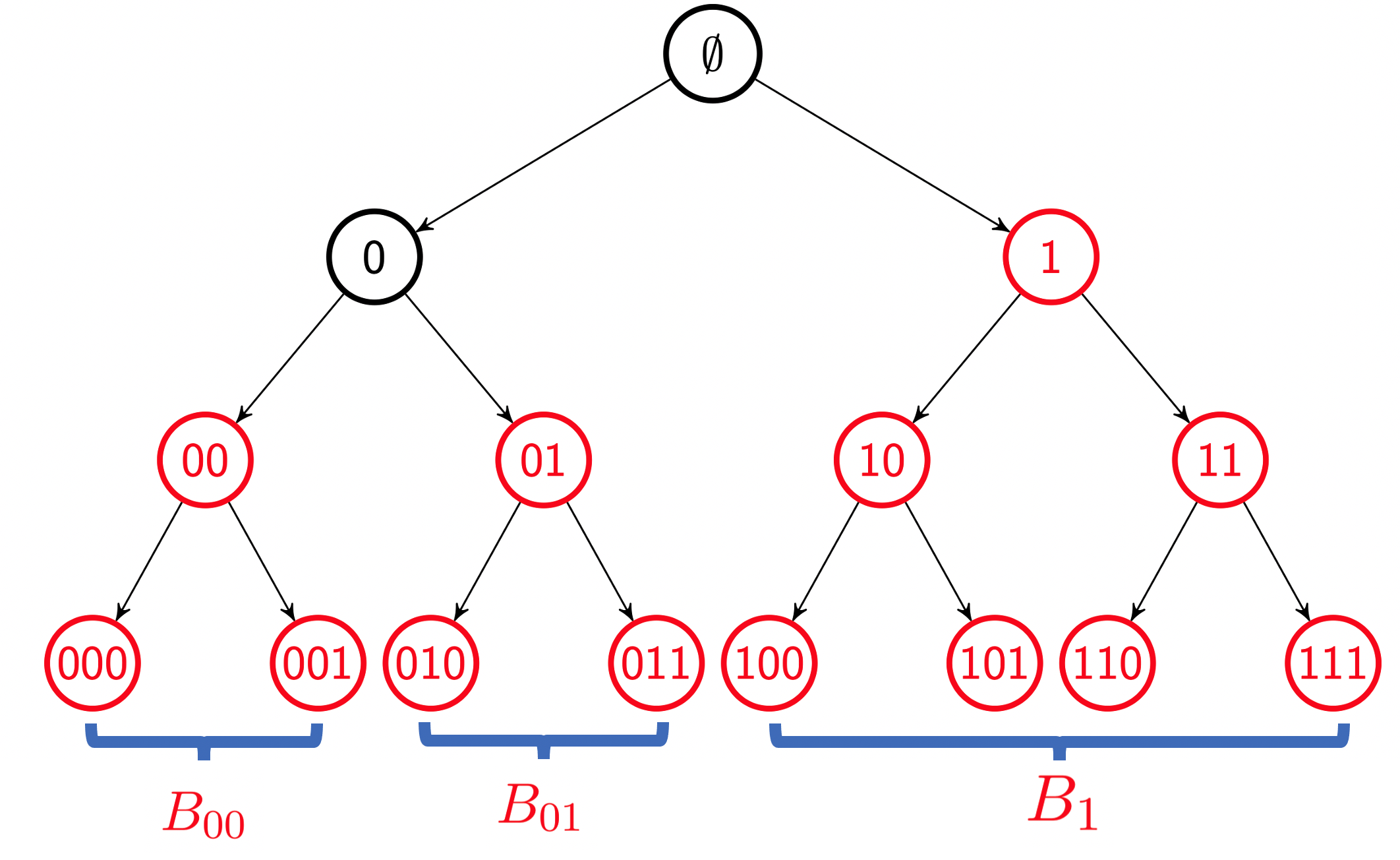}
\caption{The first partition in Lemma~\ref{lem:partition} is shown in the case that $\mathcal L_{\stable}=\{B_{00},B_{01},B_1\}$ with $(k,K)=(2,3)$. It states that $[k]_0^K=[2]_0^3=B_{00}\cup B_{01}\cup B_1$. }
\end{framed}
\label{fig:blocks}
\end{figure}

\subsection{No Edge Intersections in Expectation}

In this subsection we prove Proposition~\ref{prop:uppersimple}. As explained in the outline, the idea is to estimate $\mathbb E[|E(G,G')|]$ by a sum of individual contributions from each $x\in\mathcal L_{\stable}$ and then control the total contribution from each digit profile.

\begin{lemma}

\label{lem:bernstein}

Let $X\sim \Bin(N,q)$ for some $q\in [0,1]$. Then for $t\leq \sqrt{Nq(1-q)}$,
\[
    \mathbb P\left[\big| X-\mathbb E[X] \big| \geq t\sqrt{Nq(1-q)}\right] \leq e^{-\Omega(t^2)}
\]
holds uniformly over $q$.

\end{lemma}

\begin{proof}
This follows from Bernstein's inequality; see for instance \cite[Inequality (2.10)]{boucheron2013concentration}.
\end{proof}

\begin{lemma}\label{lem:tx}

For any $x\in [k]_0^M$, either 
\[
    \min(t_x,1-t_x)=0
\]
or 
\[
    \min(t_x,1-t_x)\asymp_{\bp} N^{-1+2\cP(x)}
\]
holds. The same holds for $\min(t_x+\lambda_x,1-t_x-\lambda_x)$. Here $\asymp_{\bp}$ denotes asymptotic equality for large $N$ up to $\bp$-dependent constant factors. 

\end{lemma}

\begin{proof}

We focus on $\min(t_x,1-t_x)$ (as the two statements are symmetric) and assume $x$ has a digit $x[i]\neq 0$ so that $t_x\neq 0$. If $x[1]=0$ and $i>1$ is minimal with $x[i]\neq 0$, then $b_0(x)\log(N)=i-1$ and so 
\[
    t_x\asymp_{\bp} p_0^{b_0(x)\log(N)}=N^{-1+2\cP}.
\] 
Similarly if $x[1]>0$ and $i'>1$ is minimal with $x[i']\neq (k-1)$, then 
\[
    1-t_x-\lambda_x\asymp_{\bp} p_{k-1}^{b_{k-1}(x)\log(N)}=N^{-1+2\cP}.
\]
\end{proof}

\begin{lemma}

Let $x\in\mathcal L_{\stable}$ have digit profile $(b_0,b_{k-1},c_0,\dots,c_{k-1})$. Then
\begin{align}
\label{eq:sizeconcentration}\mathbb P\left[\bigg| |\mathcal I(B_x)|-N^{\cL}\bigg| \geq N^{\frac{\cL+\delta}{2}}\right]&\leq e^{-\Omega(N^{\delta})},\\
\label{eq:positionconcentration}
\mathbb P\left[\big| \iota(x)-Nt_x]\big| \geq  N^{\cP+\frac{\delta}{2}}\right]
&\leq e^{-\Omega(N^{\delta})}\\
\label{eq:poscon2}
\mathbb P\left[\big| \tau(x)-N(t_x+\lambda_x)\big| \geq N^{\cP+\frac{\delta}{2}}\right]
&\leq e^{-\Omega(N^{\delta})}.
\end{align}
\end{lemma}

\begin{proof}

First, inequality~\eqref{eq:sizeconcentration} follows immediately from \eqref{eq:IBx}, by applying Lemma~\ref{lem:bernstein} with $t=N^{\delta/2}$.

For inequality~\eqref{eq:positionconcentration} we similarly recall the distribution of $\iota$ given by \eqref{eq:iota}. From Lemma~\ref{lem:tx} it follows that unless $t_x=0$ (in which case $\iota(x)=1$ with probability $1$), 
\[
    \min(t_x,1-t_x)\asymp N^{-1+2\cP}.
\]
Then Lemma~\ref{lem:bernstein} with $t=N^{\delta/2}$ completes the proof of \eqref{eq:positionconcentration} as $\frac{\delta}{2}<\min(\frac{\cL}{2},\cP)$ by Lemma~\ref{lem:Lstable}. Inequality \eqref{eq:poscon2} is proved identically.
\end{proof}

The next lemma shows that for any $i\in [N]$, there are at most two blocks $B_{x_{i,1}},B_{x_{i,2}}$ that $i$ could plausibly appear in.

\begin{lemma}
\label{lem:blocksvalid}

For each index $i\in [N]$, there exist $x_{i,1},x_{i,2}\in\mathcal L_{\stable}$ with \[\mathbb P[i\in \mathcal I(B_{x_{i,1}})\cup \mathcal I(B_{x_{i,2}})]\geq 1-e^{-\Omega(N^{\delta})}.\]

\end{lemma}

\begin{proof}

Choose $x_{i,1}\in\mathcal L_{\stable}$ so that $\frac{i}{N}\in J_x=[t_{x_{i,1}},t_{x_{i,1}}+\lambda_{x_{i,1}})$, and without loss of generality assume \[\frac{i}{N}\in \bigg[t_{x_{i,1}}+\frac{\lambda_{x_{i,1}}}{2},t_{x_{i,1}}+\lambda_{x_{i,1}}\bigg).\] Then we obtain

\begin{align*}\iota(x_{i,1})&\leq Nt_{x_{i,1}}+\left|\iota(x_{i,1})-Nt_{x_{i,1}}\right|\\
&\leq i-\frac{N\lambda_{x_{i,1}}}{2}+\left|\iota(x_{i,1})-Nt_{x_{i,1}}\right|.\end{align*}

As 
\[
    N\lambda_{x_{i,1}}=N^{\cL(x_{i,1})}\geq N^{\cP(x_{i,1})+\delta},
\]
using inequality~\eqref{eq:positionconcentration} implies that 
\[
    \mathbb P[\iota(x_{i,1})\leq i]\geq 1-e^{-\Omega(N^{\delta})}
\]
If $x_{i,1}$ is the lexicographically last element of $\mathcal L_{\stable}$ then $\iota(x_{i,1})\leq i$ already implies $i\in \mathcal I(B_{x_{i,1}})$. Otherwise using Lemma~\ref{lem:partition} we take $x_{i,2}\in\mathcal L_{\stable}$ immediately lexicographically following $x_{i,1}$, so that $t_{x_{i,1}}+\lambda_{x_{i,1}}=t_{x_{i,2}}.$ Reasoning identically to the above shows that 
\[
    \mathbb P[\tau(x_{i,2})\geq i]\geq 1-e^{-\Omega(N^{\delta})}.
\]
If $\iota(x_{i,1})\leq i \leq \tau(x_{i,2})$, then $i\in \mathcal I(B_{x_{i,1}})\cup \mathcal I(B_{x_{i,2}})$ holds because $x_{i,1}$ and $x_{i,2}$ are consecutive in $\mathcal L_{\stable}$. The result follows.
\end{proof}

Based on the previous lemma, we now upper-bound $\mathbb E[|E(G,G')|]$ by a sum over the individual blocks $B_{x}$. Recall that $E(G_{B_x})\subseteq E(G)$ is the set of edges $(i,i+1)\in E(G)$ coming from strings $s_i=s_{i+1}\in B_x$.

\begin{lemma}\label{lem:AMGM} $\mathbb E[|E(G,G')|]\leq e^{-\Omega(N^{\delta})}+ 4 \sum_{x\in\mathcal L_{\stable}} \sum_{i=1}^{N-1} \mathbb P[(i,i+1)\in E(G_{B_x})]^2.$

\end{lemma}

\begin{proof}

Lemma~\ref{lem:blocksvalid} and the AM-GM inequality imply
\begin{align*}
\mathbb E[|E(G,G')|] &\leq \sum_{i=1}^{N-1} \mathbb P[(i,i+1)\in E(G,G')]\\
&\leq e^{-\Omega(N^{\delta})}+\sum_{i=1}^{N-1} \sum_{j_1,j_2\in \{1,2\}} \mathbb P[(i,i+1)\in E(G_{B_{x_{i,j_1}}},G_{B_{x_{i,j_2}}})]\\
& \leq e^{-\Omega(N^{\delta})}+2\sum_{i=1}^{N-1} \sum_{j\in \{1,2\}} \mathbb P[(i,i+1)\in E(G_{B_{x_{i,j}}})]^2\\
&\leq e^{-\Omega(N^{\delta})}+ 4 \sum_{x\in\mathcal L_{\stable}} \sum_{i=1}^{N-1} \mathbb P[(i,i+1)\in E(G_{B_x})]^2.
\end{align*}
\end{proof}

The next lemma uses the quantity
\[
    \cW(x)=\left(\frac{M(x)-K}{\log N}\right)\psi_{\bp}(2)=\left(b_0(x)+b_{k-1}(x)+c_{\tot}(x)-\frac{K}{\log N}\right)\psi_{\bp}(2)<0
\]
defined near the end of Subsubsection~\ref{subsubsec:digit-profile}.

\begin{lemma}\label{lem:blockedges}
For any $x\in\mathcal T$, 
\[
    \mathbb E\left[|E(G_{B_x})|~\big\vert~ |\mathcal I(B_x)|\right]\leq |\mathcal I(B_x)|^2 N^{\cW(x)}.
\]
\end{lemma}

\begin{proof}

The right-hand side upper-bounds the expected number of pairs $(i,j)$ with $s_i=s_j$ and $i,j\in \mathcal I(B_x)$, by summing over the $|\mathcal I(B_x)|^2$ pairs of pre-sorted strings in $B_x$. Indeed it is easy to see that for independent $\mu_{\bp,K}$-random strings $s$ and $s'$, and fixed $x\in [k]_0^M$,
\[
    \mathbb P[s=s'|s,s'\in B_x]= \phi_{\bp}(2)^{-(K-M)}= N^{\cW(x)}.
\]
\end{proof}

The following lemma upper-bounds the probability for an edge $(i,i+1)$ to appear in $E(G_{B_x})$ as a function of $x$, uniformly over $i\in[N]$. The idea is that even conditioned on the value $|\mathcal I(B_x)|$ and the internal structure of $\mathcal I(B_x)$, the remaining randomness of the value $\iota(x)$ has a ``homogenizing'' effect.

\begin{lemma}\label{lem:pointwise}

For any $x\in \mathcal L_{\stable}$ and index $i\in [N-1]$,  

\[\mathbb P[(i,i+1)\in E(G_{B_x})]\leq 4N^{2\cL(x)-\cP(x)+\cW(x)+2\delta} +e^{-\Omega(N^{\delta})}.\]

\end{lemma}

\begin{proof}
We condition on the multiset of strings $S_x\equiv [s_j|s_j\in B_x]$ appearing in $B_x$. We will show that if 
\begin{equation}\label{eq:assume0}
    |\mathcal I(B_x)|\leq 2N^{\cL}\leq N/2
\end{equation}
holds, then
\[
    \mathbb P[(i,i+1)\in E(G_{B_x})|S_x]\leq 4N^{2\cL(x)-\cP(x)+\cW(x)+2\delta}.
\]
This implies the desired result since by inequality~\eqref{eq:sizeconcentration}, 
\[
    \mathbb P[|\mathcal I(B_x)|\leq 2N^{\cL}]\geq 1-e^{-\Omega(N^{\delta})}.
\]

Observe that the multiset $S_x$ determines the values $|E(G_{B_x})|$ and $|\mathcal I(B_x)|=|S_x|$, and in fact determines the entire set $E(G_{B_x})$ up to shifts. Given $S_x$, it is easy to see that $\iota(x)$ has conditional law 
\[
    \iota(x)\sim \Bin\left(N-|\mathcal I(B_x)|,\frac{t_x}{1-\lambda_x}\right)+1.
\] 
From \eqref{eq:assume0}, we have $N-|\mathcal I(B_x)|\geq N/2$.
Because any $x\in\mathcal L_{\stable}$ has length $\Omega(\log(N))$ by Lemma~\ref{lem:Lstable}, it follows that $\lambda_x\leq \frac{1}{2}$ for all $x\in\mathcal L_{\stable}$ when $N$ is large enough. Therefore Lemma~\ref{lem:tx} gives $t_x=0$ or $t_x\geq \Omega(N^{-1+2\cP})$. Similarly 
\[
    1-\frac{t_x}{1-\lambda_x}=\frac{1-t_x-\lambda_x}{1-\lambda_x}\geq\Omega(N^{-1+2\cP})
\] unless $1-t_x-\lambda_x=0$.

Let us now split into two cases, the first being that 
\[
    \min(t_x,1-t_x-\lambda_x)>0.
\]
In this case we conclude that $\iota(x)-1$
is binomial with number of trials $N-|\mathcal I(B_x)|\geq N/2$ and total variance $\Omega(N^{2\cP})$. Recalling that $\cP(x)\geq \delta$ for $x\in\mathcal L_{\stable}$, the Lindeberg condition implies that conditionally on $S_x$, $\iota(x)$ satisfies a central limit theorem with standard deviation $\Omega\left(N^{\cP(x)}\right)$. Since $\iota(x)-1$ is binomial, this implies a pointwise bound on its probability mass function. Explicitly, we may apply either \cite[Equation 25]{pitman1997probabilistic} or the combination of \cite[Equation 24]{pitman1997probabilistic} and \cite[Theorem B]{canfield1980application} to obtain 
\begin{equation}\label{eq:mass-bound}
    \max_{j\in [N]} \mathbb P[\iota(x)=j|S_x]\leq N^{-\cP(x)+2\delta}.
\end{equation}

Next in the second case, assume that
\[
    \min(t_x,1-t_x-\lambda_x)=0.
\]
This simply means that $x$ consists of all digits $0$ or all digits $(k-1)$. Then $c_{\tot}(x)=0$ and so $\cL=2\cP\leq \cP+2\delta$ implies $\cP\leq 2\delta.$ Hence \eqref{eq:mass-bound} holds in either case. As a result for any $i\in [N-1]$,
\begin{align*}
\mathbb P\left[(i,i+1)\in E(G_{B_x})\mid S_x\right]&\leq |E(G_{B_x})|\cdot \max_{j\in [N]}\mathbb P\left[\iota(x)=j|S_x\right]\\
&\leq |E(G_{B_x})|\cdot N^{-\cP(x)+2\delta}. 
\end{align*}
Applying Lemma~\ref{lem:blockedges} shows that when \eqref{eq:assume0} holds,
\[
\mathbb P\left[(i,i+1)\in E(G_{B_x})\mid S_x\right]\leq 4 N^{2\cL(x)-\cP(x)+\cW(x)+2\delta}.
\]
\end{proof}

Using Lemma~\ref{lem:pointwise}, we can estimate each term appearing in Lemma~\ref{lem:AMGM}.

\begin{lemma}
\label{lem:blockbound}
For any $x\in \mathcal L_{\stable}$, 
\[
    \sum_{i=1}^{N-1}\mathbb P[(i,i+1)\in E(G_{B_x})]^2\leq 64 N^{5\cL(x)-2\cP(x)+2\cW(x)+4\delta} + e^{-\Omega(N^{\delta})}.
\]
\end{lemma}

\begin{proof}

For those $i\in [N]$ with \[i\in \left[Nt_x-N^{\cP+\frac{\delta}{2}},N(t_x+\lambda_x)+N^{\cP+\frac{\delta}{2}}\right],\] Lemma~\ref{lem:pointwise} implies

\[\mathbb P[(i,i+1)\in E(G_{B_x})]\leq 4N^{2\cL(x)-\cP(x)+\cW(x)+2\delta} +e^{-\Omega(N^{\delta})}.\]

As $\cP+\frac{\delta}{2}\leq \cL-\frac{\delta}{2}$, the above applies to at most $2N^{\cL}$ values of $i$. For all other $i\in [N-1]$, inequalities~\eqref{eq:positionconcentration} and \eqref{eq:poscon2} imply $\mathbb P[(i,i+1)\in E(G_{B_x})]\leq e^{-\Omega(N^{\delta})}.$ Combining and using $(a+b)^2\leq 2a^2+2b^2$ yields
\begin{align*}
\sum_{i=1}^{N-1}\mathbb P[(i,i+1)\in E(G_{B_x})]^2&\leq 2N^{\cL}\left(4N^{2\cL(x)-\cP(x)+\cW(x)+2\delta} +e^{-\Omega(N^{\delta})}\right)^2+Ne^{-\Omega(N^{\delta})}\\
&\leq 64 N^{5\cL(x)-2\cP(x)+2\cW(x)+4\delta} + e^{-\Omega(N^{\delta})}.
\end{align*}
\end{proof}

Having controlled the individual summands in Lemma~\ref{lem:AMGM} in terms of the digit profile of $x$, it remains to sum over $x\in\mathcal L_{\stable}$. This amounts to determining the number of $x\in\mathcal L_{\stable}$ with each possible digit profile, and then finding the maximum possible contribution of each digit profile. Recalling the definition
\[
\cE=c_{\tot} H(c_0,\dots,c_{k-1})+5\cL-2\cP+2\cW,
\]
it follows from Lemma~\ref{lem:blockbound} and Proposition~\ref{prop:entropy} that the contribution of a given digit profile $x$ to the bound of Lemma~\ref{lem:AMGM} is roughly $N^{\cE(x)}$. The next lemma shows that $\cE(x)$ is uniformly negative over $x\in\mathcal L_{\stable}$ when $K\geq (\underline{C}_{\bp}+\varepsilon)\log(N)$. Here we give a concise proof which does not provide much intuition for the constants $\theta_{\bp}$ and $C_{\bp}$. See Subsection~\ref{subsec:informalnumeric} for another argument which is longer and less formal but probably more enlightening.

\begin{lemma}
\label{lem:UBnumerics}
For $\delta=\delta(\bp,\varepsilon)$ small enough, if $K\geq (\underline{C}_{\bp}+\varepsilon)\log(N)$ then
\[
    \max_{(b_0,b_{k-1},c_0,\dots,c_{k-1})\text{ }\delta\text{-stable}} \cE(b_0,b_{k-1},c_0,\dots, c_{k-1}) \leq -\Omega_{\bp}(\varepsilon)<0. 
\]
\end{lemma}

\begin{proof}
Let us extend the definitions of $c_{\tot},\cP,\cL,\cW,$ and $\cE$ to be functions of arbitrary $(k+2)$-tuples $(b_0,b_{k-1},c_0,\dots,c_{k-1})\in\left(\mathbb R^+\right)^{k+2}$ which are constrained to satisfy $\min(b_0,b_{k-1})=0$. Having done this, we observe that $\cE=\cE(b_0,b_{k-1},c_0,\dots,c_{k-1})$ is affine in $t$ along the paths 
\begin{equation}\label{eq:linear-path}
    t\in \mathbb R\to \left(\left(1-t\alpha_{\bp}\right)b_0,\left(1-t\alpha_{\bp}\right)b_{k-1},(1+t)c_0,\dots ,(1+t)c_{k-1}\right)
\end{equation} 
where $\alpha_{\bp}\geq 0$ is chosen so that $\cL-\cP$ remains constant as $t$ varies. 

 Therefore to conclude we only need to show $\cE\leq -\Omega(\varepsilon)$ at the endpoint cases, which take the forms $(b_0,b_{k-1},0,\dots,0)$ and $(0,0,c_0,\dots,c_{k-1})$ and which continue to satisfy $\cL-\cP\in [\delta,2\delta].$ As either $b_0=0$ or $b_{k-1}=0$, we assume without loss of generality that $b_{k-1}=0$. In the case $(b_0,0,\dots,0)$, we get
\begin{align*}
    \cE(b_0,0,\dots,0)
    &=5-5b_0\log\left(\frac{1}{p_0}\right)-1
    +b_0\log\left(\frac{1}{p_0}\right)
    +2\left(b_0-\frac{K}{\log(N)}\right)\psi_{\bp}(2)
    +2\delta\\
    &=4\left(1-b_0\log\left(\frac{1}{p_0}\right)\right)
    +2\left(b_0-\frac{K}{\log(N)}\right)\psi_{\bp}(2)
    +2\delta
\end{align*} 
From $\cL-\cP\in [\delta,2\delta]$ we obtain \[\cL-\cP=\frac{1-b_0\log\left(\frac{1}{p_0}\right)}{2}\in [\delta,2\delta]\] and so 
\[
b_0\log\left(\frac{1}{p_0}\right)\in [1-4\delta,1-2\delta].
\] 
Using also that 
\[
    \frac{K}{\log N}\geq \underline{C}_{\bp}+\varepsilon\geq \frac{1}{\log(1/p_0)}+\varepsilon,
\]
we find
\begin{align*}
\cE(b_0,0,\dots,0)&\leq 8\delta + 2\left(\frac{1-2\delta}{\log\left(\frac{1}{p_0}\right)}-\frac{1+\varepsilon}{\log\left(\frac{1}{p_0}\right)}\right)\psi_{\bp}(2)+2\delta\\
&\leq -\Omega_{\bp}(\varepsilon)+10\delta\\
&\leq -\Omega_{\bp}(\varepsilon).
\end{align*}
The last inequality above holds because $\delta=\delta(\bp,\varepsilon)$ is sufficiently small. We now turn to the main task of estimating $\cE(0,0,c_0,\dots,c_{k-1})$. We use the following identities and inequalities.
\begin{itemize}
    \item $\cL-\cP\in [\delta,2\delta]$.
    \item $\cP=\frac{1}{2}$.
    \item $H(\bp^{\theta_{\bp}})=\theta_{\bp} I(\bp,\bp^{\theta_{\bp}})-\psi_{\bp}(\theta_{\bp})$. 
    \item $\psi_{\bp}(\theta_{\bp})=2\psi_{\bp}(2)$.
\end{itemize}
To deal with the entropy term in $\cE$, we use the non-negativity of Kullback-Leibler divergence. For any discrete probability distribution $\mathbf{q}=(q_0,\dots,q_{k-1})$ (with $\sum_{i=0}^{k-1} q_i=1$), 
\begin{align*} 
    H(q_0,\dots,q_{k-1})
    &=
    \sum_{i=0}^{k-1} q_i \log\left(\frac{1}{(\bp^{\theta_{\bp}})_i}\right)-D_{\term{KL}}(\mathbf{q},\bp^{\theta_{\bp}}) 
    \\
    &\leq 
    \sum_{i=0}^{k-1} q_i \log\left(\frac{1}{(\bp^{\theta_{\bp}})_i}\right)
    \\
    &= 
    -\psi_{\bp}(\theta_{\bp})+\theta_{\bp}\sum_{i=0}^{k-1} q_i \log\left(\frac{1}{p_i}\right).
\end{align*}
Using the above estimate with $q_i=\frac{c_i}{c_{\tot}}$, we find
\begin{align}
    \nonumber\cE(0,0,c_0,\dots,c_{k-1})&=c_{\tot}H(c_0,\dots,c_{k-1})+5(\cL-\cP)+\frac{3}{2}+2\cW+2\delta\\
\label{eq:KLuse}
    &\leq -c_{\tot}\psi_{\bp}(\theta_{\bp})+\theta_{\bp}\sum_{i=0}^{k-1} c_i \log\left(\frac{1}{p_i}\right) +\frac{3}{2}+2\cW+12\delta\\
\nonumber
    &\leq \theta_{\bp}\sum_{i=0}^{k-1} c_i \log\left(\frac{1}{p_i}\right) +\frac{3}{2}-2(\underline{C}_{\bp}+\varepsilon)\psi_{\bp}(2)+12\delta\\
\label{eq:cE-bound}
    &\leq \theta_{\bp}\sum_{i=0}^{k-1} c_i \log\left(\frac{1}{p_i}\right) +\frac{3}{2}-2\psi_{\bp}(2)\underline{C}_{\bp}-\Omega_{\bp}(\varepsilon).  
\end{align}
The last line again follows because $\delta$ is sufficiently small. Finally we recall the following:
\begin{align*}
    \sum_{i=0}^{k-1} c_i \log\left(\frac{1}{p_i}\right)
    &=
    1-\cL=\frac{1}{2}+O(\delta),
    \\
    C_{\bp}
    &=
    \frac{3+\theta_{\bp}}{4\psi_{\bp}(2)}\leq \underline{C}_{\bp}=\max\left(C_{\bp},\frac{1}{\log(1/p_0)},\frac{1}{\log(1/p_{k-1})}\right).
\end{align*}
Substituting into the estimate \eqref{eq:cE-bound}, we obtain
\[
    \cE(0,0,c_0,\dots,c_{k-1})\leq \frac{3+\theta_{\bp}+O(\delta)}{2}-2\psi_{\bp}(2)\underline{C}_{\bp}-\Omega_{\bp}(\varepsilon)\leq -\Omega_{\bp}(\varepsilon).
\] 
This completes the proof.
\end{proof}

Proposition~\ref{prop:uppersimple} readily follows by combining the ingredients just established.

\begin{proof}[Proof of Proposition~\ref{prop:uppersimple}]

We start from the upper bound in Lemma~\ref{lem:AMGM} and group the strings $x\in\mathcal L_{\stable}$ by their digit profile. For each digit profile $(b_0,b_{k-1},c_0,c_1,\dots,c_{k-1})$, by Proposition~\ref{prop:entropy} the number of corresponding blocks $x\in\mathcal L_{\stable}$ is at most
\[
\binom{c_{\tot}\log(N)}{c_0\log(N),\dots,c_{k-1}\log(N)}\leq N^{c_{\tot} H(c_0,\dots,c_{k-1})}.
\]
Lemmas~\ref{lem:blockbound} and \ref{lem:UBnumerics} imply that for each fixed digit profile $(b_0,b_{k-1},c_0,c_1,\dots,c_{k-1})$,
\begin{align*}
\sum_{\substack{x\in\mathcal L_{\stable},\\\text{Digit Profile}(x)=(b_0,\dots,c_{k-1})}} \sum_{i=1}^{N-1} \mathbb P[(i,i+1)\in E(G_{B_x})]^2 & \leq 64 N^{c_{\tot}H(c_0,\dots,c_{k-1})+5\cL-2\cP+2\cW+2\delta}+e^{-\Omega(N^{\delta})} \\
&=64 N^{\cE+4\delta}+e^{-\Omega(N^{\delta})}\\
&\leq 64 N^{-\Omega_{\bp}(\varepsilon)}+e^{-\Omega(N^{\delta})}.
\end{align*}
Since there are at most $O(\log^{k+2}(N))\leq N^{o(1)}$ total digit profiles $(b_0,b_{k-1},\dots,c_{k-1})$, Lemma~\ref{lem:AMGM} therefore yields the desired estimate
\[
\mathbb E[|E(G,G')|]\leq 256N^{-\Omega_{\bp}(\varepsilon)}+e^{-\Omega(N^{\delta})}.
\]
\end{proof}

\subsection{Informal Derivation of the Value $C_{\bp}$}\label{subsec:informalnumeric}

We saw the constant 
\[
    \psi_{\bp}(2)=-\log\sum_{i=0}^{k-1}p_i^2
\]
arise naturally in Lemma~\ref{lem:blockedges}, expressed via $\cW$. In this informal subsection, we will explain why the constants $\theta_{\bp}$ and $C_{\bp}$ appeared in the final stages of the proof above by determining ``straightforwardly'' how large $\frac{K}{\log N}$ must be for Lemma~\ref{lem:UBnumerics} to hold. We again view $\cE(c_0,\dots,c_{k-1})$ as a continuous function and restrict to the main case that $b_0=b_{k-1}=0$. Moreover we will set all $O(\delta)$ terms to zero for simplicity. For $x\in\mathcal L_{\stable}$ with $b_0(x)=b_{k-1}(x)=0$, we have $c_L(x)= c_F(x)= 1/2$ which yields the constraint equation
\begin{equation}
\label{eq:cLapprox}
    \sum_{i=0}^{k-1} c_i\log(1/p_i)=\frac{1}{2}.
\end{equation}
Setting $C=\frac{K}{\log N}$, we find from $c_L(x)= c_F(x)= 1/2$ that
\[
\cE = \big(H(c_0,\dots,c_{k-1})+2\psi_{\bp}(2)\big)\cdot c_{\tot}+\frac{3}{2}-2C\psi_{\bp}(2).
\]
To maximize $\cE=\cE(c_0,\dots,c_{k-1})$ given the constraint \eqref{eq:cLapprox}, we set the gradient $\nabla \cE$ to be parallel to the constraint direction $\big(\log(1/p_0),\log(1/p_1),\dots,\log(1/p_{k-1})\big)$. (Without arguing too formally, one expects there are no issues of maxima occurring at the boundary because the entropy function is concave and its inward-normal derivative diverges when any coordinate approaches $0$.) By writing out the definition of entropy one readily computes that the maximizer $(c_0^*,\dots,c_{k-1}^*)$ satisfies
\begin{align*}
\theta \log(1/p_i)&=\frac{\partial~}{\partial c_i}(\cE^*) \\
&= 2\psi_{\bp}(2)+\log(c_{\tot}^*/c_i^*)-1+\sum_{j\in [k]_0} \frac{c_j^*}{c_{\tot}^*}\\
&=2\psi_{\bp}(2)+\log(c_{\tot}^*/c_i^*)
\end{align*}
for some proportionality constant $\theta\in\mathbb R$. Recalling that $\psi_{\bp}(t)=-\log \phi_{\bp}(t)$ for $\phi_{\bp}(t)=\sum_{i=0}^{k-1} p_i^t$, we obtain by rearranging
\begin{align*}
\frac{c_i^*}{c_{\tot}^*}&=e^{2\psi_{\bp}(2)}p_i^{\theta}\\
&= \frac{p_i^{\theta}}{\phi_{\bp}(2)^2}.
\end{align*}

Since $\sum_{i=0}^{k-1} \frac{c_i^*}{c_{\tot}^*}=1$ it follows that $\phi_{\bp}(\theta)=\phi_{\bp}(2)^2$, i.e. $\theta=\theta_{\bp}$. Moreover we find $\left(\frac{c_0^*}{c_{\tot}^*},\dots,\frac{c_{k-1}^*}{c_{\tot}^*}\right)=\bp^{\theta_{\bp}}.$ Solving for $c_{\tot}^*$ using \eqref{eq:cLapprox} above yields
\[
\frac{1}{c_{\tot}^*}=\frac{2}{\phi(\theta_{\bp})}\sum_{i=0}^{k-1} p_i^{\theta_{\bp}}\log(1/p_i)=2I(\bp,\bp^{\theta_{\bp}}).
\]
Finally plugging back into the definition of $\cE$ and recalling properties of $I(\bp,\bp^t)$,
\begin{align*}
\cE(c_0^*,\dots,c_{k-1}^*)&=\frac{H(\bp^{\theta_{\bp}})+2\psi_{\bp}(2)}{2I(\bp,\bp^{\theta_{\bp}})}+\frac{3}{2}-2C\psi_{\bp}(2)\\
&=\frac{H(\bp^{\theta_{\bp}})+\psi_{\bp}(\theta_{\bp})}{2I(\bp,\bp^{\theta_{\bp}})}+\frac{3}{2}-2C\psi_{\theta_{\bp}}(2)\\
&=\frac{3+\theta_{\bp}}{2}-2C\psi_{\bp}(2).
\end{align*}
Rearranging shows that $\cE^*<0$ is equivalent to
\[
C>C_{\bp}=\frac{3+\theta_{\bp}}{4\psi_{\bp}(2)}=\frac{3+\theta_{\bp}}{2\psi_{\bp}(\theta_{\bp})}.
\]
Therefore we have ``straightforwardly'' recovered the statement of Lemma~\ref{lem:UBnumerics}. Let us also point out that
\begin{align*}
c_{\tot}&=\frac{1}{2I(\bp,\bp^{\theta_{\bp}})}=\frac{\theta_{\bp}}{2(H(\bp^{\theta_{\bp}})+\psi_{\bp}(\theta_{\bp}))}\\
&<\frac{\theta_{\bp}}{2\psi_{\bp}(\theta_{\bp})}<\frac{3+\theta_{\bp}}{2\psi_{\bp}(\theta_{\bp})}\\
&=C_{\bp}\leq \underline{C}_{\bp}.
\end{align*}
Here we used \eqref{eq:IDI} in the first line. Hence the maximizer we found corresponds to ``real'' blocks $B_x$ with length $M\approx c_{\tot}\log N<K$.

Since this argument ignored $O(\delta)$ error terms and some details on boundary issues, we verified Lemma~\ref{lem:UBnumerics} directly in the previous section instead of making the informal argument rigorous. The main step of this verification was to use non-negativity of the Kullback-Leibler divergence $D_{\term{KL}}(\mathbf{q},\bp^{\theta_{\bp}})$ with $q_i=\frac{c_i}{c_{\tot}}$ in inequality~\eqref{eq:KLuse}. Given the argument above, this step becomes quite natural. Indeed $\cE$ is linear in $(c_0,\dots,c_{k-1})$ except for the entropy term, so \eqref{eq:KLuse} simply linearizes this entropy term around the equality case $\left(\frac{c_0^*}{c_{\tot}^*},\dots,\frac{c_{k-1}^*}{c_{\tot}^*}\right)\approx \bp^{\theta_p}$.

\section{Proof of Lemma~\ref{lem:tails}}\label{sec:UBfinish}

In this section we prove Lemma~\ref{lem:tails}, whose statement is recalled now.

\tails*

It was shown in Section~\ref{sec:UB} how to upper-bound the total variation distance from uniform after $K$ $\bp$-shuffles based on the exponential moment estimate above. Therefore establishing Lemma~\ref{lem:tails} will complete the proof of the mixing time upper bound \eqref{eq:UB}.

\subsection{Preparatory Lemmas}\label{subsec:prep}

Define $F(a,b)$ to be the value $\E[|E(G,G')|]$ for i.i.d. $\bp$-random shuffle graphs $G$ and $G'$ on decks of $a$ cards with $b$ shuffles. Proposition~\ref{prop:uppersimple} provides the main upper bound on $F(a,b)$, stated as a bound on $F(N,K)$. The next lemma gives a much easier estimate we will use for small values of $a$ and $b$.

\begin{lemma}\label{lem:easyUB}

For any non-negative integers $a$ and $b$,
\[
    F(a,b)
    \leq \min\left(a,a^2\cdot\phi_{\bp}(2)^b\right).
\]
\end{lemma}

\begin{proof}
The bound $F(a,b)\leq a$ is obvious. The other bound \[\E[|E(G,G')|]\leq \E[|E(G)|]\leq  a^2\phi_{\bp}(2)^b\] follows by summing over all $\binom{a}{2}$ pairs of strings $s_i,s_j$ as in Lemma~\ref{lem:blockedges}.
\end{proof}

The next two lemmas allow us to upper-bound relatively complicated expected edge intersections based on simple expected edge intersections. They will be used below to estimate the left-hand side of \eqref{eq:condint} as a sum over the blocks in the decomposition \eqref{eq:blocks}.

\begin{lemma}\label{lem:mix0}

Let $A$ and $B$ be independent random subsets of a finite set $\mathcal A$. Let $A'$ and $B'$ respectively be independent copies of $A$ and $B$. Then
\[
    \mathbb E[|A\cap B|]\leq \frac{\mathbb E[|A\cap A'|]+\mathbb E[|B\cap B'|]}{2}.
\]
\end{lemma}

\begin{proof}
For each $a\in \mathcal A$ let $A_a=\mathbb P[a\in A]$ and $B_a=\mathbb P[a\in B]$. Then the statement reduces to showing $\sum_a A_aB_a\leq \frac{\sum_a (A_a^2+B_a^2)}{2}$ which holds by AM-GM.
\end{proof}

\begin{lemma}
\label{lem:mix}
Let $A$ be a random subset of a finite set $\mathcal A$ and let $\mathcal F$ be a $\sigma$-algebra. Let $A'$ be an independent copy of $A$ and let $A_{\mathcal F}$ and $A_{\mathcal F}'$ be conditionally independent copies of $A$ conditioned on $\mathcal F$. Then
\begin{equation}\label{eq:L2}
\mathbb E[|A\cap A'|]\leq \mathbb E[|A_{\mathcal F}\cap A'_{\mathcal F}|].
\end{equation}
\end{lemma}

\begin{proof}
For each element $a\in \mathcal A$, let $Q_a=\mathbb P[a\in A|\mathcal F]$. Let $P_a=\mathbb P[a\in A]=\mathbb E[Q_a]$. Then \eqref{eq:L2} amounts to showing 
\[
    \sum_{a\in\mathcal A} P_a^2 \leq \sum_{a\in \mathcal A} \mathbb E[Q_a^2].
\]
Since $\mathbb E[Q_a^2]-P_a^2\geq 0$ is simply the variance of the random variable $Q_a$ for each $a$, the result follows.
\end{proof}

\subsection{The Edge-Exploration Process}\label{subsec:edgeexplore}

We now define the exploration process mentioned at the end of Section~\ref{sec:UB}, which explores a pair $(s_1,\dots,s_N),(s_1',\dots,s_N')\in \mathcal S$ of sorted string sequences in order starting from $s_1,s_1'$. At step $i$, the currently revealed strings are 
\[
    (s_1,\dots, s_{i}) \quad\text{and}\quad (s_1',\dots,s'_{i})
\]
which results in revealed subgraphs 
\[
    G_i\subseteq G,\quad G_i'\subseteq G'
\]
that grow with $i$. Explicitly, $G_i$ and $G_i'$ are simply the induced subgraphs of $G$ and $G'$ on the vertex set $\{1,2,\dots,i\}$. When either $s_i$ or $s_i'$ begins with the prefix $[(k-1)(k-1)]$ we stop the process. Essentially by definition, this process finds all edges in $E_{\for}(G,G')$. As alluded to at the end of Section~\ref{sec:UB}, the following lemma shows how to bound the exponential moments of $E_{\for}(G,G')$ using this exploration process.

\begin{lemma}
\label{lem:geodom}
Suppose $\gamma>0$ is such that the conditional expectation estimate
\begin{equation}\label{eq:condint}
\mathbb E[E_{\for}(G,G')-E(G_i,G_i')|\mathcal F_i]\leq \gamma
\end{equation}
holds almost surely with $\mathcal F_i\equiv \sigma(s_1,\dots,s_i,s_1',\dots,s_i')$ for each $i\in [N]$. Then 
\[
    \mathbb E[e^{t\cdot E_{\for}(G,G')}]\leq 1+2e^{t}\gamma
\]
for any $t>0$ satisfying $e^t\gamma\leq \frac{1}{10}$.
\end{lemma}

\begin{proof}
Define for simplicity the random variable $X=E_{\for}(G,G')$. For each $j\geq 0$ let $t_j=\inf\{i:E(G_i,G_i')\geq j\}$. Then $t_j$ is an stopping time, and if $t_j<\infty$ then $|E(G_{t_j},G'_{t_j})|=j$ holds because $E(G_{i+1},G_{i+1}')-E(G_i,G_i')\leq 1$ holds almost surely for each $i$. Morever when $t_j<\infty$ we have
\[
    \mathbb P[X>j|\mathcal F_{t_j}]\leq \gamma
\]
due to the assumption \eqref{eq:condint}. Of course the inequality $X\geq j$ implies that $t_j$ is finite. Hence by optional stopping, we may average the above display to conclude that
\[
    \mathbb P[X>j|X\geq j]\leq\gamma.
\] 
This means $X$ has hazard rate at least that of a geometric random variable $Y$ with
\[
    \mathbb P[Y=j]=(1-\gamma)\gamma^j,\quad j\geq 0.
\]
Therefore $X$ is stochastically dominated by $Y$. Using the assumption $e^t\gamma\leq \frac{1}{10}$, we find
\begin{align*}
    \mathbb E[e^{tX}]&\leq \mathbb E[e^{tY}]\\
    &\leq (1-\gamma)\sum_{j\geq 0} (e^t\gamma)^j\\
    &\leq\frac{1}{1-e^t \gamma}\\
    &\leq 1+2e^t \gamma.
\end{align*}
\end{proof}

To analyze the exploration process we group the potential future strings which are lexicographically larger than $s_i$. Supposing that $s_i<_{\lex}[(k-1)(k-1)]$ does not begin with $[(k-1)(k-1)]$, set 
\[
    \Blocks(s_i)=\Blocks(s_i,[(k-1)(k-1)])
\]
in the notation of Lemma~\ref{lem:segtree} just below. By construction, $\Blocks(s_i)$ consists of $O(\log N)$ blocks and
\begin{equation}
\label{eq:blocks}\{s\in [k]_0^K:s_i<_{\lex} s<_{\lex} [(k-1)(k-1)]\}= \bigcup_{x\in \Blocks(s_i)} B_{x}.
\end{equation}
The fact that $|\Blocks(s_i)|\leq O(\log N)\leq N^{o(1)}$ will be used in the proof of Lemma~\ref{lem:tails} in the next subsection. It allows us to estimate a sum over $x\in\Blocks(s_i)$ by its maximum term; see just before the start of Case 1 therein.

\begin{figure}[H]
\centering
\begin{framed}
\includegraphics[width=\linewidth]{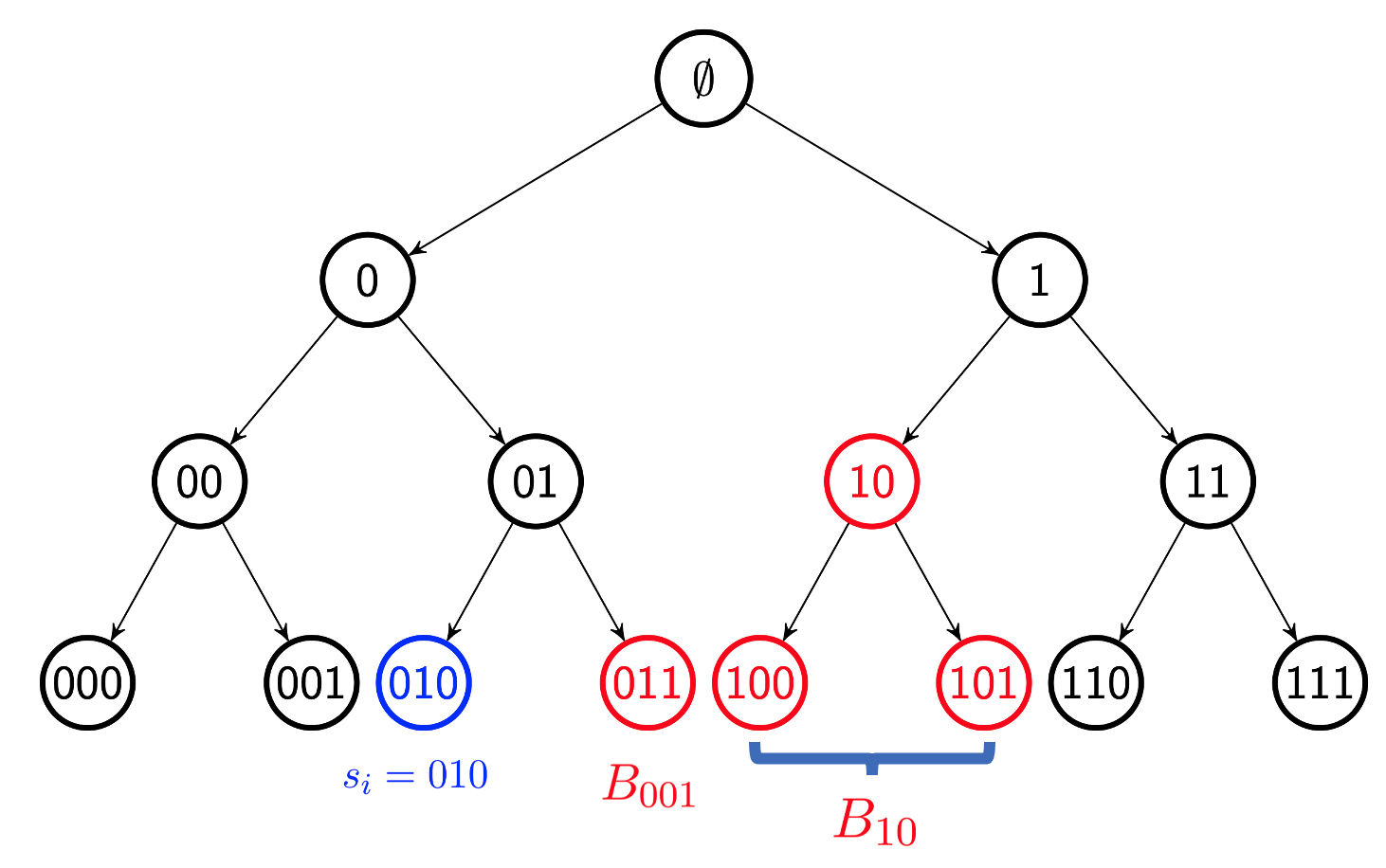}
\caption{The decomposition of \eqref{eq:blocks}, guaranteed by Lemma~\ref{lem:segtree}, is shown when $s_i=010$ with $(k,K)=(2,3)$. It states that $\{s\in [2]_0^3:010<_{\lex} s<_{\lex} 11\}=B_{001}\cup B_{10}$. }
\end{framed}
\label{fig:future}
\end{figure}

\begin{lemma}\label{lem:segtree}

Let $s_a<_{\lex}s_b$ be strings each of length at most $K$. Define the lexicographic interval 
\[
I_{s_a,s_b}\equiv\{s\in [k]_0^K:s_a<_{\lex}s<_{\lex}s_b\}.
\]
Then $I_{s_a,s_b}$ can be written as a disjoint union of blocks
\[
I_{s_a,s_b}=\bigcup_{x\in \Blocks(s_a,s_b)} B_x
\]
for some set $\Blocks(s_a,s_b)$ containing at most $2Kk\leq O(\log N)$ strings, each of length at most $K$.

\end{lemma}

\begin{proof}

For $0\leq M\leq K$, define
\[\overline{\Blocks}^M(s_a,s_b)=\{x\in [k]_0^M:B_x\cap I_{s_a,s_b}\neq\emptyset\}
\] to be the set of all length-$M$ strings $x$ such that $B_x$ has non-trivial intersection with $I_{s_a,s_b}$. Similarly define 
\[
\underline{\Blocks}^M(s_a,s_b)=\{x\in [k]_0^M:B_x\subseteq I_{s_a,s_b}\}
\] to be the set of all length-$M$ strings $x$ such that $B_x$ is contained inside $I_{s_a,s_b}$. Clearly $\underline{\Blocks}^M(s_a,s_b)\subseteq \overline{\Blocks}^M(s_a,s_b)$. Moreover the fact that $I_{s_a,s_b}$ is a lexicographic interval means these sets differ in at most $2$ elements, i.e. 
\begin{equation}\label{eq:blockdiff}
\big|\overline{\Blocks}^M(s_a,s_b)\backslash \underline{\Blocks}^M(s_a,s_b)\big|\leq 2.
\end{equation}
Define 
\begin{align*}\overline{\Blocks}(s_a,s_b)&=\bigcup_{0\leq M\leq K} \overline{\Blocks}^M(s_a,s_b),\\
\underline{\Blocks}(s_a,s_b)&=\bigcup_{0\leq M\leq K} \underline{\Blocks}^M(s_a,s_b).
\end{align*}
Next, for any $s\in I_{s_a,s_b}$, note that all ancestors (prefixes) of $s$ are contained in $\overline{\Blocks}(s_a,s_b)$, while $\emptyset\notin \underline{\Blocks}(s_a,s_b)$. Let $y_s$ be the longest ancestor string of $s$ with 
\[
    y_s\notin \underline{\Blocks}(s_a,s_b).
\]
By definition $y_s\neq s$, so $y_s$ has a child $x_s$ which is also an ancestor of $s$ (possibly $x_s=s$). By definition of $y_s$,
\[
    x_s\in \underline{\Blocks}(s_a,s_b)
\]
and so
\[
    B_{x_s}\subseteq I_{s_a,s_b}.
\]
We claim the blocks $B_{x_s}$ constructed in this way from $s\in I_{s_a,s_b}$ are pairwise equal or disjoint. Indeed if 
\[
    B_{x_s}\subsetneq B_{x_{s'}}
\]
then $x_{s'}$ is a prefix of $y_s$. However 
\[
    y_s\notin \underline{\Blocks}(s_a,s_b)
\]
and
\[
    x_{s'}\in \underline{\Blocks}(s_a,s_b)
\]
which contradicts the fact that $ \underline{\Blocks}(s_a,s_b)$ is descendent-closed. 

Above, we started from an arbitrary $s\in I_{s_a,s_b}$ and found a block $B_{x_s}$ containing $s$. It follows that the distinct blocks $B_{x_s}$ appearing in the above construction form a partition of $I_{s_a,s_b}$. Finally, note that by inequality~\eqref{eq:blockdiff}, and the fact that $y_s$ has length at most $K-1$, $y_s$ ranges over a set of size at most $2K$. Hence $x_s$ ranges over a set of size at most $2Kk$. This concludes the proof.
\end{proof}

The next lemma shows that conditioning on the exploration process for $G$ up to $s_i$ with 
\[
    s_i<_{\lex}[(k-1)(k-1)]
\]
does not dramatically increase the typical size of $\mathcal I(B_{x})$ for any $x\in \Blocks(s_i)$. The fact that $s_i<_{\lex}[(k-1)(k-1)]$ is crucial here, as was discussed at the end of Subsection~\ref{subsec:high-level}. Indeed conditioning on $s_i=[(k-1)^K]$ would imply that 
\[
    s_i=s_{i+1}=\dots=s_N=[(k-1)^K]
\]
so that $E(G)$ contains all remaining potential edges $(i,i+1),(i+1,i+2),\dots,(N-1,N)$. However when $s_i<_{\lex}[(k-1)(k-1)]$, a constant fraction of the $\mu_{\bp,K}$-measure of $[k]_0^K$ remains not yet occupied, which prevents such an example from occuring.

\begin{lemma}\label{lem:conbin}
Conditioned on $(s_1,\dots,s_i)$ which satisfy $s_i<_{\lex}[(k-1)(k-1)]$, for any $x\in \Blocks(s_i)$ the conditional distribution of $|\mathcal I(B_{x})|$ is stochastically dominated by a $\Bin(N,p_{\Min}^{-2}\lambda_x)$ random variable. 
\end{lemma}

\begin{proof}
Condition further on the largest value $j\in [N]$ with $s_i=s_j$. Then we can generate all strings $(s_{j+1},\dots,s_N)$ by sampling i.i.d. random numbers $a'_{j+1},\dots,a'_N$ uniformly from $[t_{s_i}+\lambda_{s_i},1]$, sorting them into increasing order $a_{j+1}\leq a_{j+2}\leq\dots\leq a_N$, and choosing $s_{\ell}\in[k]_0^K$ such that $a_{\ell}\in J_{s_{\ell}}$ for $\ell\geq j+1$. There are $N-j\leq N$ such random numbers $a_{\ell}$, and $1-(t_{s_i}+\lambda_{s_i})\geq p_{\Min}^{2}$ because of the assumption that $s_i<_{\lex}[(k-1)(k-1)]$. Therefore conditionally on $j$, each $a'_i$ has probability at most $p_{\Min}^{-2}\lambda_x$ to fall into the interval $J_x$, which completes the proof.
\end{proof}

\subsection{Proof of Lemma~\ref{lem:tails}}

We now complete the proof of Lemma~\ref{lem:tails}. In light of Lemma~\ref{lem:geodom} it remains to show that the conditional expectation for the number of unrevealed edges in $E_{\for}(G,G')$, given by
\[
\mathbb E[E_{\for}(G,G')-E(G_i,G_i')|\mathcal F_i],
\]
is almost surely bounded by $O(N^{-\delta})$. The idea is to use Lemmas~\ref{lem:mix0} and \ref{lem:mix} to upper-bound this quantity by a sum over the future blocks appearing in \eqref{eq:blocks}, see Equation~\eqref{eq:futuresum2} in the proof below. Analyzing the summand corresponding to a block $B_x$ for $x\in[k]_0^M$ amounts to a smaller version of the problem considered in Proposition~\ref{prop:uppersimple} since $B_x$ can be viewed as a copy of $[k]_0^{K-M}$. As a result, the summand for $B_x$ has value $F(|\mathcal I_{B_x}|,K-M)$. This term can be estimated by Lemma~\ref{lem:easyUB} when $\mathbb E[|\mathcal I_{B_x}|]\leq N^{\delta}$ is small (Cases $1$ and $2$ of the proof below) and by Proposition~\ref{prop:uppersimple} when $\mathbb E[|\mathcal I_{B_x}|]\geq N^{\delta}$ is reasonably large (Case $3$ of the proof).

\begin{proof}[Proof of Lemma~\ref{lem:tails}]

Take $\delta=\delta(\bp,\varepsilon)$ sufficiently small, $\eta=\eta(\bp,\varepsilon,\delta)$ smaller and $\zeta=\zeta(\bp,\varepsilon,\delta,\eta)$ yet smaller. Define the following $\sigma$-algebras.
\begin{align*}
\mathcal F_i&=\sigma(s_1,\dots,s_i,s_1',\dots,s_i'),\\
\widetilde{\mathcal F}_i&=\sigma\left(s_1,\dots,s_i,s_1',\dots,s_i',\left(\mathcal I(B_{x})\right)_{x\in\Blocks(s_i)}\right).
\end{align*} 
(Note that the $\sigma$-algebras $\widetilde{\mathcal F}_i$ do not define a filtration as $i$ varies.) Let 
\[
    G_{u,1}=E_{\for}(G)\backslash E(G_i)
\]
consist of all so-far-unrevealed edges which do not involve strings beginning with $[(k-1)(k-1)]$. Let $G_{u,2}$ be a conditionally independent copy of $G_{u,1}$ given $\widetilde{\mathcal F}_i$ - equivalently this means $G_{u,2}$ is obtained by resampling $G_{u,1}$ conditioned to have the same sets $\mathcal I(B_x)$ for each $x\in\Blocks(s_i)$. Define $G'_{u,1},G'_{u,2}$ the same way for $G'$. Hence $G_{u,1},G_{u,2},G'_{u,1},G'_{u,2}$ are shuffle graphs with all edge-endpoints in $\{i,i+1,\dots,N\}$. 

We will show that at any time $i$ in the exploration process, the expected number of unrevealed edges in $E_{\for}(G,G')$ is bounded by
\[
\E[|E(G_{u,1},G_{u,1}')|\big\vert \mathcal F_i]\leq O(N^{-\zeta}).
\]
By Lemma~\ref{lem:geodom}, this will complete the proof of Lemma~\ref{lem:tails} up to replacing $\zeta$ with $\delta$. First, using Lemmas~\ref{lem:mix0} and \ref{lem:mix} conditionally on $\mathcal F_i$, we estimate the expected number of unrevealed edges by
\[
\E[|E(G_{u,1},G_{u,1}')|\big\vert \mathcal F_i]\leq \E\left[\frac{|E(G_{u,1},G_{u,2})|+|E(G'_{u,1},G'_{u,2})|}{2}\big\vert \mathcal F_i \right].
\]
Therefore by symmetry it suffices to show that
\[
\E\left[|E(G_{u,1},G_{u,2})| \big\vert \mathcal F_i\right]\leq O(N^{-\zeta})
\]
holds almost surely. By definition, conditioning on $\widetilde{\mathcal F}_i$ determines the interval $\mathcal I(B_x)$ for each such $x$. Moreover the remaining $K-M$ digits of each of the $|\mathcal I(B_{x})|$ random strings in $B_x$ are still i.i.d. $\bp$-random. As a consequence, 
\begin{equation}
\label{eq:cond-decomp}
    \mathbb E[E(G_{u,1},G_{u,2})\big\vert \widetilde{\mathcal F}_i]= \left|\{j>i:s_j=s_i\}\right|+\sum_{x\in \Blocks(s_i)}F(|\mathcal I(B_{x})|,K-M).
\end{equation}
Indeed recall from the start of Subsection~\ref{subsec:prep} that $F(a,b)$ is the expected size of $E(G,G')$ when there are $a$ cards and $b$ shuffles. Thus \eqref{eq:cond-decomp} essentially holds by definition.

Next, the law of total expectation yields
\begin{align}
\nonumber
    \mathbb E[E(G_{u,1},G_{u,2})\big\vert \mathcal F_i]&=\E\left[\mathbb E[E(G_{u,1},G_{u,2})\big\vert \widetilde{\mathcal F}_i]\bigg\vert\mathcal F_i\right]\\
\label{eq:futuresum2}
    &= \E\left[\left|\{j>i:s_j=s_i\}\right|\big\vert \mathcal F_i\right]
    +
    \sum_{x\in \Blocks(s_i)}\E\big[F(|\mathcal I(B_{x})|,K-M)\big\vert \mathcal F_i\big].
\end{align}
The first term on the right-hand side of \eqref{eq:futuresum2} is controlled by Lemma~\ref{lem:connorepeat}, which implies 
\[
\mathbb E[|\{j>i:s_j=s_i\}|\big\vert \mathcal F_i]\leq O(N^{-\zeta}).
\]
To estimate the other (main) term on the right-hand side of \eqref{eq:futuresum2}, we will show for each $x\in\Blocks(s_i)$ that
\[
\E\big[F(|\mathcal I(B_{x})|,K-M)\big\vert \mathcal F_i\big]\leq O(N^{-\zeta}).
\]
As $|\Blocks(s_i)|=O(\log N)\leq N^{o(1)}$ this suffices to finish the proof. We now split into three cases depending on the size of $\lambda_x$. In all cases below we let $M$ denote the length of $x$. Case $3$ (the main one) is where Proposition~\ref{prop:uppersimple} is essential.

Let us emphasize that $|I(B_{x})|$ is still random conditionally on $\mathcal F_i$. While we do not have good almost sure bounds on $|I(B_{x})|$ itself, its conditional distribution is uniformly stochastically bounded by Lemma~\ref{lem:conbin}. Since we are estimating a conditional expectation given $\mathcal F_i$ and not $\widetilde{\mathcal F}_i$, this suffices for an almost sure bound.

\proofstep{Case 1: $\lambda_x\leq N^{-1-\delta}$.} In this case, Lemmas~\ref{lem:easyUB} and~\ref{lem:conbin} imply 
\begin{align*}
    \E\big[F(|\mathcal I(B_{x})|,K-M)|\mathcal F_i\big]
    &\leq 
    \mathbb E[|\mathcal I(B_{x})|]
    \\
    &\leq 
    O(N^{-\zeta}).
\end{align*}

\proofstep{Case 2: $N^{-1-\delta}\leq\lambda_x\leq N^{-1+\delta}$.} 
In this case, Lemmas~\ref{lem:conbin} and~\ref{lem:bernstein} imply that $|\mathcal I(B_{x})|\leq N^{2\delta}$ holds with probability $1-e^{-\Omega(N^{\delta})}$. The fact $\lambda_x\leq (p_{\M})^M$ implies
\begin{align*}
    M
    &\leq 
    \frac{\log (\lambda_x^{-1})}{\log(p_{\M}^{-1})}
    \\
    &\leq 
    \frac{ (1+\delta) \log N}{\log(p_{\M}^{-1})}
 \end{align*} 
In particular as $\delta\ll\varepsilon$ is sufficiently small this implies $K-M\geq \Omega_{\bp}(\varepsilon) \log N$. Lemma~\ref{lem:easyUB} now yields
\begin{align*}
    \E\big[F(|\mathcal I(B_{x})|,K-M)|\mathcal F_i\big]&\leq \mathbb E[|\mathcal I(B_{x})|^2] \phi_{\bp}(2)^{\Omega_{\bp}(\varepsilon)\log(N)} 
    \\
    &\leq 
    O\left(N^{2\delta-\Omega_{\bp}(\varepsilon)}\right)
    \\
    &\leq 
    O(N^{-\zeta}).
\end{align*}

\proofstep{Case 3: $\lambda_x\geq N^{-1+\delta}$.} Similarly to the previous case, observe that

\begin{align}M&\leq \frac{\log(\lambda_x^{-1})}{\log(p_{\M}^{-1})}\\
&\leq \overline{C}_{\bp}\log(\lambda_x^{-1}).\label{eq:almostdone}\end{align}

We break into subcases depending on $|\mathcal I(B_{x})|$. The first subcase is that $|\mathcal I(B_{x})|\leq N^{\eta}$. Here the lower bound $K-M\geq \Omega_{\bp}(\delta \log N)$ follows from inequality~\eqref{eq:almostdone}, and applying Lemma~\ref{lem:easyUB} yields
\[F(|\mathcal I(B_x)|,K-M)\leq N^{2\eta} \phi_{\bp}(2)^{K-M}\leq N^{-\Omega_{\bp}(\delta)}.
\]

In the main subcase $|\mathcal I(B_{x})|\in [N^{\eta},2p_{\Min}^{-2} N\lambda_x]$ we obtain:

\begin{align}\label{eq:ratio} K-M&\geq  (\overline{C}_{\bp}+\varepsilon)\log(N\lambda_x)\\
\nonumber&\geq\left(\overline{C}_{\bp}+\frac{\varepsilon}{2}\right) \log(2p_{\Min}^{-2}N\lambda_x)\\
\nonumber&\geq \left(\overline{C}_{\bp}+\frac{\varepsilon}{2}\right) \log|\mathcal I(B_{x})|.
\end{align}
Since $|\mathcal I(B_x)|\geq N^{\eta}$ tends to infinity with $N$, Proposition~\ref{prop:uppersimple} implies 
\[
    F(|\mathcal I(B_{x})|,K-M)\leq O\left(|\mathcal I(B_x)|^{-\delta}\right)\leq  O(N^{-\zeta}).
\] 

 Finally the subcase $|\mathcal I(B_{x})|\geq 2p_{\Min}^{-2} N\lambda_x$ occurs with tiny probability $e^{-\Omega(N^{\delta})}$ by Lemmas~\ref{lem:conbin} and~\ref{lem:bernstein}. In this subcase we use the trivial bound $F(|\mathcal I(B_{x})|,K-M)\leq N$. Combining subcases, we have established that whenever Case 3 holds, 
 \[
    \E\big[F(|\mathcal I(B_{x})|,K-M)|\mathcal F_i\big]\leq O(N^{-\zeta}).
\] 
Combining cases (and substituting $\delta$ for $\zeta$ at the end) concludes the proof of Lemma~\ref{lem:tails}.
\end{proof}

\begin{remark}\label{rem:gap}

Recall that throughout Section~\ref{sec:proof}, and in particular in Proposition~\ref{prop:uppersimple}, the weaker inequality $K\geq (\underline{C}_{\bp}+\varepsilon)\log N$ sufficed where
\[
    \underline{C}_{\bp}\equiv\max\left(C_{\bp},\frac{1}{\log(1/p_0)},\frac{1}{\log(1/p_{k-1})}\right)\leq \overline{C}_{\bp}.
\]
This means that when $k>2$, for some parameter choices such as $\bp=(0.01,0.98,0.01)$, the expectation $\E[|E(G,G')|]$ becomes small before mixing occurs, so the exponential moments of $|E(G,G')|$ are still large. This discrepancy can be explained as follows. When $K$ satisfies
\[
\underline{C}_{\bp}+\varepsilon<\frac{K}{\log N}<\overline{C}_{\bp}-\varepsilon, 
\]
the graph $G$ typically contains $N^{\Omega(1)}$-size connected components coming from strings with nearly all digits $i_{\M}$. In such situations $\E[|E(G,G')|]\leq o(1)$ is small by Proposition~\ref{prop:uppersimple}. However an easy pigeonhole argument on $N$ copies of $G$ shows that with $\Omega(1/N^2)$ probability, $E(G,G')$ contains an $N^{\Omega(1)}$-sized component formed by a large $G$-component and large $G'$-component overlapping. As a result $|E(G,G')|$ has large exponential moments. (Moreover this argument still applies if we initially require $S,S'\in\mathcal S_1$ to be ``typical''.)

In upper-bounding the mixing time, the bound $K\geq (\overline{C}_{\bp}+\varepsilon)\log N$, as opposed to $K\geq (\underline{C}_{\bp}+\varepsilon)\log N$, is necessary in two places. The first is in Lemma~\ref{lem:connorepeat}. The other occurs above in \eqref{eq:ratio} where we needed to ensure that Proposition~\ref{prop:uppersimple} yields an upper bound for $F(|\mathcal I(B_x)|,K-M)$. In the worst case, all $M$ of $x$'s digits might be $i_{\M}$. Then typically (at least when the right-hand side below is positive),
\[
\log |\mathcal I(B_x)|\approx \log(N)-M \log(1/p_{\M}).
\]
To apply Proposition~\ref{prop:uppersimple}, we thus need
\[
K-M\geq \underline{C}_{\bp}(\log N-M \log(1/p_{\M}))
\]
to hold for any $M$ making both sides positive. In particular, if we continuously increase $M$ the right side must reach $0$ before the left side, which implies $K\geq \frac{\log N}{\log(1/p_{\M})}.$ On the other hand, when $M=0$ we need $K\geq \underline{C}_{\bp}\log N$ for Proposition~\ref{prop:uppersimple} to apply. Hence at least in bounding the exponential moments of $|E(G,G')|$, the value $\overline{C}_{\bp}=\max\left(\underline{C}_{\bp},\frac{1}{\log(1/p_{\M})}\right)$ arises from the need to apply Proposition~\ref{prop:uppersimple} for all sizes of block $B_x$ appearing in the partition \eqref{eq:blocks}.

\end{remark}

\section{Proof of the Mixing Time Lower Bound}

\label{sec:lb}

In this section we take $K=\lfloor (C_{\bp}-\varepsilon)\log(N)\rfloor$ and show that almost no total-variation mixing occurs after $K$ shuffles. First, when $K\leq (\widetilde C_{\bp}-\varepsilon)\log(N)$ we previously argued at the start of Subsection~\ref{subsec:intuition} that the total variation distance from uniform is $1-o(1)$. Hence we may assume that $\widetilde C_{\bp}< C_{\bp}$ holds, else there is nothing to prove. By taking $\varepsilon$ small enough, we may further assume
\begin{equation}\label{eq:assume} K \geq (\widetilde C_{\bp}+\varepsilon)\log N.
\end{equation}

For a set $H\subseteq \mathbb Z$, its boundary $\partial H\subseteq H$ is defined by
\[
\partial H\equiv\{h\in H:h-1\notin H\text{ or }h+1\notin H\}.
\]
Its edge set $E(H)$ is the set of edges with both endpoints in $H$, i.e. we identify $H$ with the corresponding induced subgraph of $G$. We will verify the following criterion from \cite{lalley2000rate} for non-mixing. The idea of the proof is to use the number of ascents of $\sigma\in\Sym_N$ within $H$ to distinguish the uniform distribution $\sigma=\pi$ from the shuffled distribution $\sigma=\pi^G$. 

\begin{proposition}[{{\cite[Proposition 2]{lalley2000rate}}\label{prop:lowerboundcond}}]
Let $(K_N)_{N\geq 1}$ be a deterministic sequence of positive integers.
Suppose there exist deterministic subsets $H=H_N\subseteq [N]$ such that for some $\delta=\delta(\bp,\varepsilon)$ the following properties hold as $N\to\infty$, where $G$ is the shuffle graph for a deck of $N$ cards undergoing $K_N$ $\bp$-shuffles:
\begin{align}
|H|&\to \infty\\
|\partial H|&=O(|H|^{1/2})\\
\mathbb P\left[|E(G)\cap E(H)|\geq |H|^{\frac{1}{2}+\delta}\right]&\to 1.
\end{align}
Then asymptotically no total-variation mixing occurs after $K_N$ shuffles, i.e.
\[\lim_{N\to\infty} d_N(K_N)=1.\]

\end{proposition}

\begin{remark}\label{rem:LB1}

By using AM-GM or Cauchy--Schwarz similarly to the proof of Lemma~\ref{lem:mix}, the conditions of Proposition~\ref{prop:lowerboundcond} imply 
\begin{align*}
\E[|E(G,G')|]&\geq \frac{\E[|E(G)\cap E(H)|]^2}{|E(H)|}\cdot (1-o(1))\\
&\geq \Omega(|H|^{2\delta})\\
&\gg 1.
\end{align*}
However it does \textbf{not} follow from what we show that $K= (\underline{C}_{\bp}\pm o(1))\log N$ is always the cutoff point where the expected number $\E[|E(G,G')|]$ of shared edges in $G$ and $G'$ transitions from superconstant to subconstant. This is because the analysis of this section assumes inequality \eqref{eq:assume}.

\end{remark}

\subsection{Preparation and Proof Idea}

Define $\alpha_{\tot}\log(N)=\left\lfloor\frac{1-\delta}{2 I(\bp,\bp^{\theta_{\bp}})}\log(N)\right\rfloor$, where as usual $\delta=\delta(\bp,\varepsilon)$ is sufficiently small. Choose (via some rounding procedure) positive integers $\alpha_0\log(N),\dots,\alpha_{k-1}\log(N)$ satisfying
\begin{equation}
\label{eq:alpha-def}
    \sum_{i=0}^{k-1} \alpha_i=\alpha_{\tot}\quad\text{ and }\quad\left|\alpha_i\log(N)-\frac{\alpha_{\tot}\log(N) p_i^{\theta_{\bp}}}{\phi_{\bp}(\theta_{\bp})}\right|\leq 1.
\end{equation}
Note that $\alpha_{\tot}\leq \frac{3\underline{C}_{\bp}}{4}$; indeed we showed in Proposition~\ref{prop:imp} that $\theta_{\bp}\leq 4$, hence
\begin{align*}
\alpha_{\tot}+O(\delta)&=\frac{1}{2I(\bp,\bp^{\theta_{\bp}})}=\frac{\theta_{\bp}}{2(H(\bp^{\theta_{\bp}})+\psi_{\bp}(\theta_{\bp}))}\\
&<\frac{\theta_{\bp}}{2\psi_{\bp}(\theta_{\bp})}<\frac{3+\theta_{\bp}}{3\psi_{\bp}(\theta_{\bp})}\\
&=\frac{2C_{\bp}}{3}\leq \frac{2\underline{C}_{\bp}}{3}.
\end{align*}
We may therefore take $\beta_{\tot}\log(N)=K-\alpha_{\tot}\log(N)\geq\Omega(\log N)$ and choose positive integers $(\beta_i\log(N))_{i\in [k]_0}$ with
\[
    \sum_{i=0}^{k-1} \beta_i=\beta_{\tot}\quad\text{ and }\quad \left|\beta_i\log(N)- \frac{\beta_{\tot}\log(N) p_i^{2}}{\phi_{\bp}(2)}\right|\leq 1.
\]

The numbers just constructed satisfy
\[
    \sum_{i=0}^{k-1}\alpha_i\log(N) + \sum_{i=0}^{k-1}\beta_i\log(N)
    = \alpha_{\tot}\log(N)+\beta_{\tot}\log(N)
    = 
    K.
\]
We will consider $G$-edges coming from strings with $\alpha_i \log(N)$ digits $i$ in the first $\alpha_{\tot}\log(N)$ digits, and $\beta_i\log(N)$ digits $i$ in the last $\beta_{\tot}\log(N)$ digits, for each $i\in [k]_0$. This is essentially a two-part digit profile. Let us point out that strings with many leading $0$ or $(k-1)$ digits will \textbf{not} require special care in this part.

\begin{definition}

The length $\alpha_{\tot}\log(N)$ string $x\in [k]_0^M$ is a \textbf{collision-likely prefix} (we write $x\in \Pre$) if $x$ contains $\alpha_i\log(N)$ digits of $i$ for each $i\in [k]_0$. 

\end{definition}

\begin{definition}
\label{def:cl}
The string $s\in [k]_0^K$ is \textbf{collision-likely} (we write $s\in \CL$) if $s$ satisfies the following properties.
\begin{itemize}
    \item With $M=\alpha_{\tot}\log(N)$, the first $M$ digits of $s$ form a collision-likely prefix.
    \item $s[M+1]=0$, $s[M+2]=1$.
    \item The $\beta_{\tot}\log(N)$ digits $s[M+1],s[M+2],\dots,s[K]$ consist of $\beta_i\log(N)$ digits of $i$ for each $i\in [k]_0$. 
\end{itemize}
\end{definition}

Recall from \eqref{eq:Jx} the definition $J_x=[t_x,t_x+\lambda_x)$ and set
\[
    H\equiv \mathbb Z\cap \left(\bigcup_{x\in \Pre} NJ_x\right).
\]
That is, $H$ consists of the ``expected locations'' of collision-likely prefixes. The set $H$ is essentially the same as in the lower bound of \cite{lalley2000rate}. Our analysis differs from Lalley's in the last part of Definition~\ref{def:cl} where we consider strings whose later digits have empirical distribution $\bp^2$.

Before proceeding into more technical details, let us give some intuition both for the definitions above and the remainder of the proof. Based on Subsection~\ref{subsec:informalnumeric}, we expect that the bulk of the edges in $E(G,G')$ come from the blocks $B_x$ with digit profile
\[
    c_i(x)\approx c_i^* = \frac{1}{2 I(\bp,\bp^{\theta_{\bp}})}\cdot \frac{\bp^{\theta_{\bp}}_i}{\phi_{\bp}(\theta_{\bp})}.
\]
Therefore we took $\alpha_i\approx c_i^*$ and defined $H$ so that
\[
    H\approx \bigcup_{x\in\Pre}\mathcal I(B_x).
\]

The main difficulty in applying Proposition~\ref{prop:lowerboundcond} is to verify the last condition by lower-bounding the number of $G$-edges appearing in blocks $B_x$ for $x\in\Pre$. Intuitively, to count these edges one should simply count pairs of strings in $B_x$ as in Lemma~\ref{lem:blockedges}. However this will overestimate the number of $G$-edges for strings that appear many times. Hence one would like to also control for example the number of equal triples $s_i=s_{i+1}=s_{i+2}=s$ --- this is relevant for obtaining the correct first moment and also for controlling the variance. Such a strategy was carried out in \cite[Lemmas 8 and 9]{lalley2000rate}. However for this approach to work, $\bp$ must be close to a uniform distribution so that the expected number of triples does not overwhelm the expected number of pairs.

Instead of counting pairs of equal strings $s_i=s_j$, we consider for each $s\in\CL$ the event $Y_s$ that $s_i=s_{i+1}=s$ holds for \textbf{at least} one $i\in [N]$. Because of the ``extra margin'' afforded by the second property of $\CL$ in Definition~\ref{def:cl}, it follows that with high probability, all $i\in [N]$ with $s_i\in \CL$ satisfy $i\in H$ (see Lemma~\ref{lem:KS}). Under this event, we have
\begin{equation}
\label{eq:trunc}
    |E(G)\cap E(H)|\geq \sum_{s\in \CL} 1_{Y_s}.
\end{equation}

The sum $\sum_{s\in \CL} 1_{Y_s}$ turns out to concentrate nicely while retaining almost the same expectated value. Indeed the indicator functions $1_{Y_s}$ are pairwise anti-correlated as $s\in \CL$ varies. Therefore whenever the expected value $\E\big[\sum_{s\in \CL} 1_{Y_s}\big]\geq N^{\Omega(1)}$ is large, Chebychev's inequality immediately implies a high-probability lower bound of the same order.  


Since $\mathbb P[Y_s]$ is a function of the digit profile of $s$, it suffices to focus on a single digit profile, keeping in mind that the prefix should be collision-likely. Restricting the sums above to $s\in\CL$ exactly corresponds to such a choice of digit profile. The reason to choose $\bp^2$ for the distribution of the later digits in the definition of $\CL$ is that conditioned on two $\bp$-random digits being equal, the distribution of this shared digit is $\bp^2$. Thus we expect most collisions inside a block $B_x$ to have digit distribution $\bp^2$ in the later $K-M$ digits.

In summary, the lower bound \eqref{eq:trunc} essentially involves two separate truncation steps. The first step, truncating $|E(G)\cap E(H)|$ to a sum of indicators $1_{Y_s}$, is important to obtain control of the second moment. The second step, restricting this sum to collision-likely strings $s\in\CL$, is simply a convenient way to isolate the dominant contribution to the sum over all strings $s$ with collision-likely prefixes $s[1]\dots s[\alpha_{\tot}\log(N)]=x\in\Pre$.

We conclude this subsection with two lemmas, the second of which verifies the ``easy'' parts of Proposition~\ref{prop:lowerboundcond}.

\begin{lemma}\label{lem:algebrawarmup}
For sufficiently large $N$,
\[
\sum_i \alpha_i \log(p_i)= \frac{-1+\delta}{2}\pm o(1).
\]
\end{lemma}

\begin{proof}

By the definition of $\alpha_i$ in \eqref{eq:alpha-def} and of $I(\bp,\bp^{\theta_{\bp}})$ in \eqref{I-def},
\begin{align}
\nonumber
    \sum_i \alpha_i \log(p_i)
    &\geq 
    \frac{(1-\delta)}{2I(\bp,\bp^{\theta_{\bp}})}\cdot \sum_i \frac{p_i^{\theta_{\bp}}\log(p_i)}{\phi_{\bp}(\theta_{\bp})}-O\left(\frac{1}{\log N}\right)\\
\label{eq:alpha-bound}
    &=
    \frac{-1+\delta}{2}-O\left(\frac{1}{\log N}\right).
\end{align}
\end{proof}

\begin{proposition}
\label{prop:lbsecond}
As $N\to\infty$ we have $|H|\to\infty$ and $|\partial H|=O(|H|^{\frac{1}{2}})$. More precisely 
\[
    |H|=N^{1+\sum_{i=0}^{k-1} \alpha_i\log(p_i)+\alpha_{\tot}H(\alpha_0,\dots,\alpha_{k-1})+o(1)}.
\]
\end{proposition}

\begin{proof}
For each $x\in \Pre$, Lemma~\ref{lem:algebrawarmup} shows
\[
    \lambda_x=N^{\sum_{i=0}^{k-1} \alpha_i\log(p_i) \pm o(1)}=N^{\frac{-1+\delta}{2}\pm o(1)}.
\]
Moreover it is easy to see that
\begin{equation}
\label{eq:NJx}
    \lfloor N\lambda_x\rfloor\leq |\mathbb Z\cap  NJ_x|\leq \lceil N\lambda_x\rceil.
\end{equation}
This immediately implies $|H|\to\infty$ as $\Pre$ is non-empty. For the precise asymptotics, Proposition~\ref{prop:entropy} implies
\[
    |\Pre|=\binom{\alpha_{\tot}\log(N)}{\alpha_0\log(N),\dots,\alpha_{k-1}\log(N)}=N^{\alpha_{\tot}H(\alpha_0,\dots,\alpha_{k-1})+o(1)}.
\]
As the discrete sets $(\mathbb Z\cap NJ_x)_{x\in \Pre}$  are disjoint, they have total size at most $N$. By \eqref{eq:NJx} these sets individually have size $N^{\frac{1+\delta}{2}+o(1)}$, and so $|\Pre|\leq N^{\frac{1-\delta}{2}+o(1)}$. This means the number of connected components of $H$ is smaller than the size of each component, hence $|\partial H|=O(|H|^{\frac{1}{2}}).$
\end{proof}

\subsection{Lower Bounding the Number of $G$-Edges Inside $H$}

It remains to show that $H$ contains many $G$-edges with high probability. The next lemma shows that with high probability, all appearances of collision-likely strings are inside $H$, so that it suffices to simply count edges $(i,i+1)$ with $s_i=s_{i+1}\in\CL$. The reason is simply that the requirements $s[M+1]=0$ and $s[M+2]=1$ effectively refine collision-likely prefixes $x\in \Pre$ to $x01$. $B_{x01}$ is deep enough inside $B_x$ to overcome the small fluctuations of $\mathcal I(B_x)$ vs $NJ_x$.

\begin{lemma}\label{lem:KS}

With probability $1-o(1)$, all $i\in[N]$ with $s_i\in \CL$ satisfy $i\in H$.

\end{lemma}

\begin{proof}

The Dvoretzky–Kiefer–Wolfowitz--Massart inequality \cite{dvoretzky1956asymptotic, massart1990tight} implies that with probability $1-o(1)$, all $y\in [k]_0^M$ for $0\leq M\leq K$ simultaneously satisfy
\begin{equation}
\label{eq:kolmogorov}\left|\iota(y)-Nt_y\right|\leq N^{\frac{1}{2}+\frac{\delta}{10}},\quad\quad\left|\tau(y)-N(t_y+\lambda_x)\right|\leq N^{\frac{1}{2}+\frac{\delta}{10}}.
\end{equation}

 We assume the inequalities~\eqref{eq:kolmogorov} hold for all $y$ and show the conclusion under this assumption. Fixing a collision-likely string $s$ with collision-likely prefix $x$, we apply \eqref{eq:kolmogorov} with $y=x$ and $y=x01$. Here $x01$ denotes concatenation. 
 Using \eqref{eq:alpha-bound}, we obtain
 \[
    \min\left(\lambda_x,\lambda_{x01},\lambda_{x1}\right)\geq \lambda_x p_{\Min}^2\geq  \Omega\left(N^{\frac{-1+\delta}{2}}\right).
 \] 
Therefore
\begin{align*}
    N(t_{x01}-t_x)
    &= N\lambda_{x01} 
    \geq 
    \Omega\left(N^{\frac{1+\delta}{2}}\right),
    \\
    N\big(t_x+\lambda_x-t_{x01}-\lambda_{x01}\big)
    &= 
    N\lambda_{x1} \geq \Omega\left(N^{\frac{1+\delta}{2}}\right).
\end{align*} 
By the triangle inequality,
\begin{align*}
    \iota(x01)&\geq Nt_x+N(t_{x01}-t_x)-\left|\iota(x01)-t_{x01}\right| \\
    &\geq Nt_x+\Omega\left(N^{\frac{1+\delta}{2}}\right)-N^{\frac{1}{2}+\frac{\delta}{10}}\\
    &\geq Nt_x 
\end{align*}
and 
\begin{align*}
    \tau(x01)&\leq N\big(t_x+\lambda_x\big)+N\big(t_{x01}+\lambda_{x01}-t_x-\lambda_x\big)+\left|\tau(x01)-t_{x01}-\lambda_{x01}\right| \\
    &\leq N\big(t_x+\lambda_x\big)-\Omega\left(N^{\frac{1+\delta}{2}}\right)-N^{\frac{1}{2}+\frac{\delta}{10}}\\
    &\leq N\big(t_x+\lambda_x\big).
\end{align*}

Altogether if ~\eqref{eq:kolmogorov} holds for all $y$, then all $x\in\Pre$ satisfy \[Nt_x\leq \iota(x01)\leq \tau(x01)\leq N(t_x+\lambda_x).\] Therefore $s_i\in B_{x01}$ implies $i\in H$, which completes the proof.
\end{proof}

Define the constant
\[
\gamma\equiv 2+2\sum_{i=0}^{k-1} (\alpha_i+\beta_i) \log(p_i)+\alpha_{\tot}H\left(\alpha_0,\dots,\alpha_{k-1}\right)+\beta_{\tot} H(\beta_0,\dots,\beta_{k-1}).
\]
We next give another important numerical lemma, which up to $O(\delta)$ terms will ensure that the number $N^{\gamma}$ of edges in $H$ is large enough for Proposition~\ref{prop:lowerboundcond} to apply. (It is only important that $\frac{\psi_{\bp}(2)}{2}\varepsilon$ is positive below.)

\begin{lemma}\label{lem:identity}

With $\alpha_i,\beta_i$ and $\gamma$ as defined above,
\begin{equation}
\label{eq:lemidentity}\gamma \geq \frac{1}{2}\left(1+\sum_{i=0}^{k-1} \alpha_i \log(p_i) + \alpha_{\tot}H\left(\alpha_0,\dots,\alpha_{k-1}\right)\right) + \frac{\psi_{\bp}(2)}{2}\varepsilon.
\end{equation}

\end{lemma}

\begin{proof}[Proof of Lemma~\ref{lem:identity}]

Recall the following definitions and identities.
\begin{itemize}
    \item $\psi_{\bp}(t)=-\log\phi_{\bp}(t)=-\log\left(\sum_{i=0}^{k-1} p_i^t\right)>0$ for any $t>1$.
    \item $\psi_{\bp}(\theta_{\bp})=2\psi_{\bp}(2).$
    \item $C_{\bp}=\frac{3+\theta_{\bp}}{4\psi_{\bp}(2)}=\frac{3+\theta_{\bp}}{2\psi_{\bp}(\theta_{\bp})}.$
    \item $I(\bp,\bp^t)=-\sum_i \frac{p_i^{t}\log(p_i)}{\phi_{\bp}(t)}.$
    \item $H(\bp^{t})=t I(\bp,\bp^{t})-\psi_{\bp}(t)$ for any $t>0$.
    \item $\alpha_{\tot}=\frac{1-\delta}{2 I(\bp,\bp^{\theta_{\bp}})}\pm o(1).$
    \item $\alpha_{\tot}+\beta_{\tot}\leq C_{\bp}-\varepsilon.$
    \item $\alpha_i=(\bp^{\theta_{\bp}})_i\cdot \alpha_{\tot}\pm o(1)$
    \item $\beta_i=(\bp^{2})_i\cdot \beta_{\tot}\pm o(1)$

\end{itemize}

After rearranging \eqref{eq:lemidentity} and multiplying by $2$, it suffices to show
\[
3+\sum_{i=0}^{k-1} (3\alpha_i+4\beta_i) \log(p_i)+\alpha_{\tot}H\left(\alpha_0,\dots,\alpha_{k-1}\right)+2\beta_{\tot} H(\beta_0,\dots,\beta_{k-1})\stackrel{?}{\geq} \psi_{\bp}(2)\varepsilon.
\]
First, replacing both entropy terms using $H(\bp^{t})=t I(\bp,\bp^{t})-\psi_{\bp}(t)$ and then substituting $\psi_{\bp}(\theta_{\bp})=2\psi_{\bp}(2)$ reduces us to showing
\begin{align*}
3+\sum_{i=0}^{k-1} (3\alpha_i+4\beta_i) \log(p_i)+\alpha_{\tot}(\theta_{\bp} I(\bp,\bp^{\theta_{\bp}})-2\psi_{\bp}(2))&+2\beta_{\tot} (2I(\bp,\bp^2)-\psi_{\bp}(2))
\\
&\stackrel{?}{\geq} \psi_{\bp}(2)\varepsilon.
\end{align*}
Using $\alpha_{\tot}+\beta_{\tot}=\frac{K}{\log N}\leq C_{\bp}-\varepsilon$, it remains to prove
\[
3+\sum_{i=0}^{k-1} (3\alpha_i+4\beta_i) \log(p_i)+\theta_{\bp} \alpha_{\tot} I(\bp,\bp^{\theta_{\bp}})+4\beta_{\tot} I(\bp,\bp^2) -2\psi_{\bp}(2)C_{\bp}\stackrel{?}{\geq} -\psi_{\bp}(2)\varepsilon.
\]
Substituting $C_{\bp}=\frac{3+\theta_{\bp}}{4\psi_{\bp}(2)}$ and $\alpha_{\tot}=\frac{1-\delta}{2 I(\bp,\bp^{\theta_{\bp}})}+o(1)$ we are reduced to showing 
\begin{equation}\label{eq:lastone}
\frac{3}{2}+\sum_{i=0}^{k-1} (3\alpha_i+4\beta_i) \log(p_i)+4\beta_{\tot} I(\bp,\bp^2)\stackrel{?}{\geq}  -\psi_{\bp}(2)\varepsilon+O(\delta)+o(1).
\end{equation}
Now, using $I(\bp,\bp^{\theta_{\bp}})=-\sum_i \frac{p_i^{\theta_{\bp}}\log(p_i)}{\phi_{\bp}(\theta_{\bp})}$ allows us to simplify 
\[
\sum_{i}\alpha_i \log(p_i) = \alpha_{\tot} \sum_i \frac{p_i^{\theta_{\bp}}\log(p_i)}{\phi_{\bp}(\theta_{\bp}) }+o(1)= -\frac{1-\delta}{2}+o(1).
\]
Furthermore,
\[
\beta_{\tot} I(\bp,\bp^2)=-\beta_{\tot}\sum_{i=0}^{k-1} \frac{p_i^2\log(p_i)}{\phi_{\bp}(2)}=-\sum_i \beta_i \log(p_i)+o(1).
\]
Substituting these near-equalities into \eqref{eq:lastone}, it suffices to show
\[ 
0\stackrel{?}{\geq} -\psi_{\bp}(2)\varepsilon+O(\delta)+o(1).
\]
Recalling that $\delta=\delta(\bp,\varepsilon)$ was chosen sufficiently small completes the proof.
\end{proof}

\begin{lemma}\label{lem:final}

With probability $1-o(1)$, at least $N^{\gamma-\delta}$ distinct $s\in\CL$ appear $2$ or more times in the $\bp$-random sequence $S=(s_1,\dots,s_N)\in\mathcal S$.

\end{lemma}

\begin{proof}

By Proposition~\ref{prop:entropy}, there are
\[
    |\CL|=N^{\alpha_{\tot}H\left(\alpha_0,\dots,\alpha_{k-1}\right)+\beta_{\tot} H(\beta_0,\dots,\beta_{k-1})+o(1)}
\]
collision-likely strings, each of which occurs $\Bin\left(N,N^{\sum_{i=0}^{k-1} (\alpha_i+\beta_i) \log(p_i)}\right)$ times in $S$. Because $(\widetilde C_{\bp}+\varepsilon)\log N\leq K$ holds (recall \eqref{eq:assume}) and $\log(p_i)\leq \log(p_{\M})<0$ for all $i$, we obtain
\begin{align*}
\sum_{i=0}^{k-1} (\alpha_i+\beta_i) \log(p_i)&\leq \frac{K\log(p_{\M})}{\log N}\\
&\leq (\widetilde C_{\bp}+\varepsilon) \log(p_{\M})\\
&\leq -1-\delta
\end{align*}
for $\delta=\delta(\bp,\varepsilon)$ sufficiently small. This implies 
\[
\left(1-N^{\sum_{i=0}^{k-1} (\alpha_i+\beta_i) \log(p_i)}\right)^N=\Omega(1).
\]
Next for each $s\in\CL$, let $Y_s$ denote the event that $s$ appears at least twice in $S$. By the binomial distribution formula, each $s\in\CL$ satisfies
\[ 
\mathbb P[Y_s]\geq\binom{N}{2}N^{2\sum_{i=0}^{k-1} (\alpha_i+\beta_i) \log(p_i)}\cdot \Omega(1)= N^{2+2\sum_{i=0}^{k-1} (\alpha_i+\beta_i) \log(p_i)+o(1)}.
\]
Letting $Y_{\tot}=\sum_{s\in \CL}1_{Y_s}$ and estimating $|\CL|$ with Proposition~\ref{prop:entropy}, we get
\[ 
\mathbb E[Y_{\tot}]\geq N^{\gamma-o(1)}.
\]
We claim that the Bernoulli random variables $(1_{Y_s})_{s\in\CL}$ are pairwise non-positively correlated, i.e.
\[
\mathbb P[Y_s\cap Y_{s'}] \leq \mathbb P[Y_s]\cdot \mathbb P[Y_{s'}],\quad s\neq s'.
\]

Indeed for any collision-likely strings $s\neq s'$, set $n_{s'}\in \mathbb Z_{\geq 0}$ to be the number of $i$ such that $s_i=s'$. It is easy to see that $\mathbb P[Y_s|n_{s'}]$ is decreasing in $n_{s'}$, which implies the claim.

From Lemmas~\ref{lem:algebrawarmup} and~\ref{lem:identity} it follows that $\gamma>\frac{1}{4}$ for $N$ large enough. Therefore for large $N$, we have
\[
    \E[Y_{\tot}]\geq \Omega(N^{1/4}).
\]
As argued just above, $Y_{\tot}$ is a sum of Bernoulli random variables $Y_s$ with pairwise non-positive correlations, which implies that $Y_{\tot}$ has smaller variance than expectation. In particular 
\[
    \text{Var}(Y_{\tot})\leq \mathbb E[Y_{\tot}]\leq O\left(\frac{\E[Y_{\tot}]^2}{N^{1/4}}\right).
\]
Chebychev's inequality now completes the proof as
\begin{align*}
    \mathbb P\left[Y_{\tot} \geq N^{\gamma-\delta}\right]
    &\geq 
    \mathbb P\left[Y_{\tot} \geq \frac{1}{2}\cdot\E\left[Y_{\tot}\right]\right] 
    \\
    &\geq 1-\frac{4\cdot\text{Var}(Y_{\tot})}{\E[Y_{\tot}]^2}
    \\
    &\geq 1-O(N^{-1/4}).
\end{align*}
\end{proof}

Based on the preceding results we finally establish the lower bound \eqref{eq:LB} on the mixing time in Theorem~\ref{thm:main}.

\begin{proof}[Proof of \eqref{eq:LB}]
By Lemmas~\ref{lem:identity} and \ref{lem:final}, with probability $1-o(1)$ at least $N^{\gamma-\delta}\geq |H|^{\frac{1}{2}+\delta}$ strings $s\in\CL$ appear at least twice in $S$. Each such $s$ by definition results in an edge $(i,i+1)\in E(G)$ with $s_i=s_{i+1}=s$. Moreover Lemma~\ref{lem:KS} implies that with probability $1-o(1)$, all of these edges appear inside $H$. Then by Lemma~\ref{lem:identity},
\[
|E(G)\cap E(H)|\geq |H|^{\frac{1}{2}+\Omega_{\bp}(\varepsilon)}\geq |H|^{\frac{1}{2}+\delta}
\]
also holds with probability $1-o(1)$. Combined with Proposition~\ref{prop:lbsecond}, it follows that $H$ satisfies the conditions of Proposition~\ref{prop:lowerboundcond}. This completes the proof.
\end{proof}

\section*{Acknowledgement}

We thank Persi Diaconis, Steve Lalley and the anonymous referee for helpful suggestions. This work was supported by NSF and Stanford graduate fellowships.

\small

\bibliographystyle{alpha}
\bibliography{main}

\end{document}